\documentclass[11pt]{amsart}
\usepackage{amssymb,latexsym}
\newtheorem{theorem}{Theorem}[section]
\newtheorem{lemma}[theorem]{Lemma}
\newtheorem{proposition}[theorem]{Proposition}

\theoremstyle{definition}
\newtheorem{definition}[theorem]{Definition}

\theoremstyle{remark}

\numberwithin{equation}{section}

\newcommand{\blankbox}[2]

\begin{document}
\setlength{\baselineskip}{1.2\baselineskip}
\title  [Neumann problem  for  elliptic equations]
{ The Neumann problem for a class of fully nonlinear elliptic partial differential equations}
\author{Bin Deng}
\address{Department of Mathematics\\
         University of Science and Technology of China\\
         Hefei, 230026, Anhui Province, China.}
\email{bingomat@mail.ustc.edu.cn}

\thanks{$*$ Research supported by NSFC No.11721101 and No.11871255. I would like to thank professor Xi-Nan Ma,
my advisor, for his constant encouragement and guidance.}
\begin{abstract}
  In this paper, we establish a global $C^2$ estimates to the Neumann problem for a class of fullly nonlinear elliptic equations. By the method of continuity, we establish the existence theorem of $k$-admissible
solutions of the Neumann problems.
\end{abstract}

\maketitle

%

%

\section{Introduction}
In this paper, we consider the $k$-admissible solutions of the Neumann problem of the fully nonlinear equations
\begin{equation}\label{eq}
S_{k}(W)=f(x),\quad\text{in}\quad \Omega,
\end{equation}
where the matrix $W=(w_{\alpha_{1}\cdot\cdot\cdot\alpha_{m},\beta_{1}\cdot\cdot\cdot\beta_{m}})_{C^{m}_{n}\times C^{m}_{n}}$ ,
for $2\leq m\leq n-1$ and $C_{n}^{m}=\frac{n!}{m!(n-m)!}$, with the elements as follows,
\begin{eqnarray}\label{w0}
w_{\alpha_{1}\cdot\cdot\cdot\alpha_{m},\beta_{1}\cdot\cdot\cdot\beta_{m}}=\sum^{m}_{i=1}\sum^{n}_{j=1}u_{\alpha_{i} j
}\delta^{\alpha_{1}\cdot\cdot\cdot\alpha_{i-1}j\alpha_{i+1}\cdot\cdot\cdot\alpha_{m}}_{\beta_{1}\cdot\cdot\cdot
\beta_{i-1}\beta_{i}\beta_{i+1}\cdot\cdot\cdot\beta_{m}},
\end{eqnarray}
a linear combination of $u_{ij}$, where $u_{ij}=\frac{\partial^2u}{\partial x_i\partial x_j}$ and $\delta^{\alpha_{1}\cdot\cdot\cdot\alpha_{i-1}\gamma\alpha_{i+1}\cdot\cdot\cdot\alpha_{m}}_{\beta_{1}\cdot\cdot\cdot
\beta_{i-1}\beta_{i}\beta_{i+1}\cdot\cdot\cdot\beta_{m}}$ is the generalized Kronecker symbol. All indexes $i,j,\alpha_i,\beta_i,\cdots$ come from $1$ to $n$. $f\in C^{\infty}(\Omega)$ is a positive function. And for any
$k=1, 2, \cdots, C_{n}^{m}$,
\begin{eqnarray}
  S_{k}(W)=S_{k}\big(\lambda(W)\big)=\sum_{1\leq i_{1}<i_{2}<\cdots<i_{k}\leq C_{n}^{m}}\lambda_{i_{1}}\lambda_{i_{2}}\cdots\lambda_{i_{k}},\nonumber
\end{eqnarray}
where $\lambda(W)=(\lambda_{1}, \lambda_{2}, \cdots, \lambda_{C_{n}^{m}})$ is the eigenvalues of $W$. We also set $S_{0}(W)=1$.

In fact, the matrix $W$ comes from the following operator $U^{[m]}$ as
in \cite{cns2} and \cite{hmw1}. First, we note
that $(u_{ij})_{n\times n}$ induces an operator $U$ on $\mathbb{R}^n$ by
\begin{eqnarray}
  U(e_{i})=\sum_{j=1}^{n}u_{i j}e_{j},\quad \forall 1\leq i\leq n,\nonumber
\end{eqnarray}
where $\{e_1,e_2,\cdots,e_n\}$ is the standard basis of $\mathbb{R}^n$.  We further extend $U$ to acting on the real
 vector space $\wedge^m\mathbb{R}^n$ by
\begin{eqnarray}
  U^{[m]}(e_{\alpha_1}\wedge\cdots\wedge e_{\alpha_m})=
  \sum_{i=1}^{m}e_{\alpha_1}\wedge\cdots\wedge U(e_{\alpha_i})\wedge\cdots\wedge e_{\alpha_m},\nonumber
\end{eqnarray}
where $\{e_{\alpha_1}\wedge\cdots\wedge e_{\alpha_m}\ | \ 1\leq \alpha_1<\cdots<\alpha_m\leq n\}$ is the standard basis
for $\wedge^m\mathbb{R}^n$. Then $W$ is the matrix of $U^{[m]}$ under this standard basis.
It is convenient to denote the multi-index by $\overline{\alpha}=(\alpha_1\cdots\alpha_m)$.
 We only consider the admissible multi-index,
 that is, $1\leq\alpha_1<\alpha_2,\cdots<\alpha_m\leq n$.
 By the dictionary arrangement, we can arrange all admissible multi-indexes
 from $1$ to $C_n^m$, and use $N_{\overline{\alpha}}$ denote the order number of
 the multi-index $\overline{\alpha}=(\alpha_1\cdots\alpha_n)$, i.e.,
 $N_{\overline{\alpha}}=1$ for $\overline{\alpha}=(12\cdots m)$, $\cdots$.
 We also use $\overline{\alpha}$ denote the index set $\{\alpha_1,\cdots,\alpha_n\}$.
It is not hard to see that
\begin{eqnarray}
 W_{N_{\overline{\alpha}}N_{\overline{\alpha}}}= w_{\overline{\alpha}\overline{\alpha}}
 =\sum_{i=1}^{m}u_{\alpha_i\alpha_i},\label{w1}
\end{eqnarray}
and
\begin{eqnarray}
 W_{N_{\overline{\alpha}}N_{\overline{\beta}}}= w_{\overline{\alpha}\overline{\beta}}
 =(-1)^{|i-j|}u_{\alpha_i\beta_j},\label{w2}
\end{eqnarray}
if the index set $\{\alpha_{1},\cdot\cdot\cdot,\alpha_{m}\}\setminus\{\alpha_i\}$ equals to the index set $\{\beta_{1},\cdot\cdot\cdot,\beta_{m}\}\setminus\{\beta_j\}$ but $\alpha_i\neq \beta_j$
; and also
\begin{eqnarray}
 W_{N_{\overline{\alpha}}N_{\overline{\beta}}}= w_{\overline{\alpha}\overline{\beta}}
 =0,\label{w3}
\end{eqnarray}
if the index sets $\{\alpha_{1},\cdot\cdot\cdot,\alpha_{m}\}$ and $\{\beta_{1},\cdot\cdot\cdot,\beta_{m}\}$
are differed by more than one elements.
It follows that $W$ is symmetric and is diagonal if $(u_{ij})_{n\times n}$ is diagonal.
The eigenvalues of  $W$ are the sums of eigenvalues of $(u_{ij})_{n\times n}$.

Define the Garding's cone in $\mathbb{R}^n$ as
\begin{eqnarray}
\Gamma_{k} = \{\mu\in\mathbb{R}^{n}|\  S_{i}(\mu) > 0, \forall 1 \leq i \leq k\}.\nonumber
\end{eqnarray}
Then we define the generalized Garding's cone as, $1\leq m\leq n$, $1\leq k\leq C_n^m$,
\begin{eqnarray}
  \Gamma_k^{(m)}=\{\mu\in\mathbb{R}^n|\ \{\mu_{i_1}+\cdots+\mu_{i_m}|\ 1\leq i_1<\cdots<i_m\leq n\}\in\Gamma_k \ \text{in}\  \mathbb{R}^{C_n^m}\}.\nonumber
\end{eqnarray}
Obviously, $\Gamma_k=\Gamma_k^{(1)}$ and $\Gamma_n\subset\Gamma_k^{(m)}\subset\Gamma_1$.
 If the eigenvalues of $D^2u$, denoted by $\mu(D^2u)$, is contained in $\Gamma_k^{(m)}$ for any $x\in\Omega$,
 then equivalently $\lambda(W)\in\Gamma_{k}$, such that the equation (\ref{eq}) is elliptic (see \cite{cns2} or \cite{l2}).
 It is naturally to define $k$-admissible solution as follows.
\begin{definition}
  We say $u$ is  $k$-admissible if $\mu(D^2u)\in \Gamma_k^{(m)}$.
  In addition, if $u$ is a solution of (\ref{eq}), we say $u$ is a $k$-admissible
  solution.
\end{definition}

If $m = 1$, (\ref{eq}) is known as the k-Hessian equation. In particular, (\ref{eq}) is the
Poisson
equation if $k = 1$, and the Monge-Amp\`ere equation if $k = n$, $m = 1$.

For the Dirichlet problem in $\mathbb{R}^{n}$, many results are known. For example, the Dirichlet
problem of Laplace equation is studied in \cite{gt}, Caffarelli-Nirenberg-Spruck \cite{cns} and Ivochkina
\cite{ivo} solved the Dirichlet problem of Monge-Amp\`ere equation, and Caffarelli-Nirenberg-Spruck \cite{cns2} solved the Dirichlet problem of general Hessian equations even including the case considered here. For the general Hessian
quotient equation, the Dirichlet problem is solved by Trudinger in \cite{tru}. Finally, Guan \cite{guanbo} treated the Dirichlet
problem for general fully nonlinear elliptic equation on the Riemannian manifolds without any geometric restrictions to the boundary.

Also, the Neumann or oblique derivative problem of partial differential equations was
widely studied. For a priori estimates and the existence theorem of Laplace equation
with Neumann boundary condition, we refer to the book \cite{gt}. Also, we can see the
book written by Lieberman \cite{l} for the Neumann or oblique derivative problem of linear
and quasilinear elliptic equations. In 1987, Lions-Trudinger-Urbas solved the Neumann
problem of Monge-Amp\`ere equation in the celebrated paper \cite{ltu}. For the the Neumann
problem of k-Hessian equations, Trudinger \cite{tru2} established the existence theorem when
the domain is a ball, and he conjectured (in \cite{tru2}, page 305) that one can solve the problem
in sufficiently smooth uniformly convex domains. Recently, Ma and Qiu \cite{mq} gave a positive
answer to this problem and solved the the Neumann problem of k-Hessian equations in
uniformly convex domains. After their work, the research on the Neumann problem of other equatios
has made many progresses(see \cite{mx} \cite{cz} \cite{cmw} \cite{w}).

For general $m$, the $W$-matrix is quite related to the ``$m$-convexity" or ``$m$-positivity" in
differential geometry and partial differential equations. We say a $C^2$ function $u$ is $m$-convex if
the sum of any $m$ eigenvalues of its Hessian is nonnegative, equivalently, $\mu(D^2u)\in\overline{\Gamma^{(m)}_{C_n^m}}$
or $\lambda(W)\in\overline{\Gamma_{C_n^m}}$.
Similarly,
we can formulate the notion of $m$-convexity for curvature operator and second fundamental
forms of hypersurfaces. There are large amount literature in differential geometry on this
subject. For example, Sha \cite{sh} and Wu \cite{wu} introduced the $m$-convexity of the sectional
curvature of Riemannian manifolds and studied the topology for these manifolds. In a series interesting papers, Harvey and Lawson
\cite{hl1} \cite{hl2} \cite{hl3}  introduce some generally convexity on the solutions of the nonlinear elliptic Dirichlet
problem, $m$-convexity is a special case. Han-Ma-Wu \cite{hmw1} obtained  an existence theorem of $m$-convex starshaped hypersurface
with prescribed mean curvature.
 More recently, in the complex space $\mathbb{C}^n$ case, Tosatti and Weinkove\cite{tw} \cite{tw2} solved the Monge-Amp\`ere equation for
 $(n-1)$-plurisubharmonic functions on a compact K\"ahler manifold, where the $(n-1)$-plurisubharmonicity means
the sum of any $n-1$ eigenvalues of the complex Hessian  is nonnegative.

From the above geometry and analysis reasons, it is naturally to study the Neumann problem for general equation (\ref{eq}).

The methods of Ma and Qiu \cite{mq} for the problem with $m=1$ can be generalized to our case. The key ingredient
in the present paper is to understand the structure of $W$, precisely, to replace the eigenvalues of $D^2u$ by the sums of them.
For $k\leq C_{n-1}^{m-1}=\frac{m}{n}C_n^m$,
we obtain an existence theorem of the $k$-admissible solution with less geometric
restrictions to the boundary. For $m<\frac{n}{2}$ and $ k=C_{n-1}^{m-1}+k_0\leq \frac{n-m}{n}C_n^m$, we can obtain an existence theorem
if $\Omega$ is strictly $(m,k_0)$-convex (see Definition  \ref{def}). It seems that as the degree of nonlinearity of the equation (\ref{eq})
increases, i.e., $k$ becomes larger,
the problem becomes more difficult to solve. Particularly, for $m=n-1$,
we get the existence of the $k$-admissible solution for $k\leq n-1$
only except that of the strictly $(n-1)$-convex solution for $k=n$.
The author will continue to study this case in \cite{deng}.

 A $C^2$ domain $\Omega\subset\mathbb{R}^n$ is convex, that is,  $\kappa_i(x) \geq 0$
   for any $x\in \partial\Omega$ and $i = 1,\cdots ,n- 1$, or equivalently,
   $\kappa(x)\in \overline{\Gamma_{ n-1}}$
   for any $x \in \partial\Omega$, where
$\kappa(x) = (\kappa_1 ,\cdots ,\kappa_{n-1} )$ denote the principal curvatures
of $\partial\Omega$ with respect to its inner normal $-\nu$. Then, we say $\Omega$ is a
strictly $k$-convex domain if $\kappa(x)\in \Gamma_{k}$.
To state the results in precise way, we need a definition of $(m,k_0)$-convexity as follows.
\begin{definition}\label{def}
 We say $\Omega$ is a
strictly $(m,k_0)$-convex if
$\kappa(x) = (\kappa_1 ,\cdots ,\kappa_{n-1} )\in \Gamma_{k_0}^{(m)}$
for any $x\in\partial\Omega$. Obviously, $\Gamma_{n-1}\subset\Gamma_{k_0}^{(m)}$ in
 $\mathbb{R}^{n-1}$, if $k_0\leq n$.
\end{definition}

We now state the main results of this paper as follows.
The case $k\leq C_{n-1}^{m-1}$ is easy to treat so we consider that first.
\begin{theorem}\label{th1.1}
  Suppose $\Omega\subset\mathbb{R}^{n}\ (n\geq3)$ is a bounded domain  with $C^{4}$ boundary,
   $2\leq m\leq n-1$ and
 $2\leq k\leq C_{n-1}^{m-1}$. Denote $\nu(x)$  the outer unit normal vector, and
 $\kappa_{min}(x)$ the minimum principal curvature at $x\in\partial\Omega$.
  Let $f \in C^{2}(\Omega)$ is a positive function, and $a, b\in C^{3}(\partial\Omega)$
   with $a>0$, $a+2\kappa_{min}>0$. Then there exists a unique $k$-admissible solution
   $u \in C^{3,\alpha}(\overline{\Omega})$ of the Neumann problem
\begin{equation}\label{eq1}
  \left\{
  \begin{aligned}
  S_{k}(W)&=f(x),\quad\text{in}\ \Omega,\\
  u_{\nu}&=-a(x)u+b(x),\quad\text{on}\ \partial\Omega.
  \end{aligned}
  \right.
\end{equation}
\end{theorem}
For $ k=C_{n-1}^{m-1}+k_0\leq \frac{n-m}{n}C_n^m$, we can settle more cases
 if $\Omega$ is strictly $(m,k_0)$-convex
as in the following theorem.
\begin{theorem}\label{th0}
  Suppose $\Omega\subset\mathbb{R}^{n}\ (n\geq3)$ is a strictly $(m,k_0)$-convex bounded
   domain  with $C^{4}$ boundary, $2\leq m\leq \frac{n}{2}$ and
 $ k=C_{n-1}^{m-1}+k_0\leq \frac{n-m}{n}C_n^m$. Denote $\nu(x)$  the outer unit normal
 vector, and $\kappa_{min}(x)$ the minimum principal curvature at $x\in\partial\Omega$.
 Let $f \in C^{2}(\Omega)$ is a positive function, and $a, b\in C^{3}(\partial\Omega)$
 with $a>0$, $a+2\kappa_{min}>0$. Then there exists a unique $k$-admissible solution
 $u \in C^{3,\alpha}(\overline{\Omega})$ of the Neumann problem
\begin{equation}\label{eq2}
  \left\{
  \begin{aligned}
  S_{k}(W)&=f(x),\quad\text{in}\ \Omega,\\
  u_{\nu}&=-a(x)u+b(x),\quad\text{on}\ \partial\Omega.
  \end{aligned}
  \right.
\end{equation}
\end{theorem}

The rest of this paper is arranged as follows. In section \ref{sec2}, we give some basic properties
of the elementary symmetric functions. In section \ref{sec3} and section \ref{sec4},
we establish $C^0$ estimates and the gradient estimates, interior and global. Specifically,
we extend the interior gradient estimates in Chou and Wang \cite{cw} to our cases.
In section \ref{sec5}, we show the proof of the global estimates of second order derivatives.
Finally, we can prove the existence theorem by the method of continuity in section \ref{sec6}.

\section{Preliminary}\label{sec2}

In this section, we give some basic properties of elementary symmetric functions.

First, we denote by
$S_{k}(\lambda|i)$ the symmetric function with $\lambda_{i} = 0$ and $S_{k}(\lambda|ij)$ the symmetric function with
$\lambda_{i} = \lambda_{j} = 0$.
\begin{proposition}\label{pro1}
  Let $\lambda=(\lambda_{1}, \cdots, \lambda_{n})\in \mathbb{R}^{n}$ and $k=1, \cdots, n$, then
  \begin{eqnarray}
    &&\sigma_{k}(\lambda)=\sigma_{k}(\lambda|i)+\lambda_{i}\sigma_{k-1}(\lambda|i),\quad \forall 1\leq i\leq n,\\
    &&\sum_{i=1}^{n}\lambda_{i}\sigma_{k-1}(\lambda|i)=k\sigma_{k}(\lambda),\\
    &&\sum_{i=1}^{n}\sigma_{k}(\lambda|i)=(n-k)\sigma_{k}(\lambda).
  \end{eqnarray}
\end{proposition}

We denote by $S_{k}(W|i)$ the symmetric function with $W$ deleting the $i$-row and
$i$-column and $S_{k}(W|ij)$ the symmetric function with $W$ deleting the $i, j$-rows and $i, j$-columns.
 We also define the mixed symmetric functions as follows, for $A=(a_{ij})_{n\times n}$, $B=(b_{ij})_{n\times n}$, $0\leq l\leq k\leq n$,
\begin{eqnarray}
  S_{k,l}(A,B)=\frac{1}{k!}\sum \delta_{j_1\cdots j_{k-l}j_{k-l+1}\cdots j_k}^{i_1\cdots i_{k-l}i_{k-l+1}\cdots i_k}
  a_{i_1j_1}\cdots a_{i_{k-l}j_{k-l}}b_{i_{k-l+1}j_{k-l+1}}\cdots b_{i_k j_k},\nonumber
\end{eqnarray}
where $\delta_{j_1\cdots j_{k-l}j_{k-l+1}\cdots j_k}^{i_1\cdots i_{k-l}i_{k-l+1}\cdots i_k}$ is the Kronecker symbol. It is easy to see that
\begin{eqnarray}
  S_k(A+B)=\sum_{i=0}^kC_k^iS_{k,i}(A,B),\label{2.0}
\end{eqnarray}
where $C_k^i=\frac{k!}{i!(k-i)!}$.
Then we have the following identities.

\begin{proposition}\label{pro2}
  Suppose $A=(a_{ij})_{n\times n}$ is diagonal, and $k$ is a positive integer, then
  \begin{equation}\label{s1}
    \frac{\partial S_{k}(A)}{\partial a_{ij}}=
    \left\{
    \begin{aligned}
    &S_{k-1}(A|i),&&\quad\text{if} \ i=j,&\\
    &0,&&\quad\text{if} \ i\neq j.&
    \end{aligned}
    \right.
  \end{equation}
  Furthermore, suppose $W=(w_{\overline{\alpha}\overline{\beta}})_{C_n^m\times C_n^m}$ defined as in (\ref{w0}) is diagonal, then
  \begin{equation}\label{sw1}
    \frac{\partial S_{k}(W)}{\partial u_{ij}}=
    \left\{
    \begin{aligned}
    &\sum_{i\in\overline{\alpha}}S_{k-1}(W|N_{\overline{\alpha})},&&\quad\text{if} \ i=j,&\\
    &0,&&\quad\text{if} \ i\neq j.&
    \end{aligned}
    \right.
  \end{equation}
\end{proposition}
\begin{proof}
  For (\ref{s1}), see a proof in  \cite{l2}.

  Note that
  \begin{eqnarray}
    \frac{\partial S_k(W)}{\partial u_{ij}}=\sum_{\overline{\alpha},\overline{\beta}}\frac{\partial S_k(W)}{\partial w_{\overline{\alpha}\overline{\beta}}}\frac{\partial w_{\overline{\alpha}\overline{\beta}}}{\partial u_{ij}},
  \end{eqnarray}
   Using (\ref{w1}), (\ref{w2}), and (\ref{w3}), (\ref{sw1}) is immediately a consequence of (\ref{s1}).
\end{proof}

  Recall that the Garding's cone is defined as
  \begin{eqnarray*}
    \Gamma_{k}=\{\lambda\in\mathbb{R}^{n} |\  S_{i}(\lambda)>0, \forall \ 1\leq i\leq k\}.
  \end{eqnarray*}
\begin{proposition}\label{pro3}
  Let $\lambda \in\Gamma_{k}$ and $k\in\{1, 2, \cdots, n\}$. Suppose that
  \begin{eqnarray*}
    \lambda_{1}\geq\cdots\geq\lambda_{k}\geq\cdots\geq\lambda_{n},
  \end{eqnarray*}
  then we have
  \begin{eqnarray}
     &&S_{k-1}(\lambda|n)\geq \cdots \geq S_{k-1}(\lambda|k) \geq \cdots \geq S_{k-1}(\lambda |1) >0,\label{2.1}\\
     &&\lambda_1 \geq \cdots \geq \lambda_k >0, \quad S_{k-1}(\lambda|k)\geq C(n,k) S_k(\lambda),\label{2.2}\\
     &&\lambda_1 S_{k-1} (\lambda |1) \geq \frac{k}{n} S_k(\lambda),\label{2.3}\\
     &&S_{k}^{\frac{1}{k}}(\lambda) \ \text{is concave in} \ \Gamma_k.\label{2.8}
  \end{eqnarray}
  where $C^k_n=\frac{n!}{k!(n-k)!}$ and $C(n,k)$ is a
  positive constant depends only on $n$ and $k$.
\end{proposition}
\begin{proof}
  All the properties are well known. For example,
  see \cite{l2} or \cite{hs} for a proof of (\ref{2.1}), \cite{LiT} for (\ref{2.2}),
  \cite{cw} or \cite{hmw} for (\ref{2.3}) and \cite{cns2} for (\ref{2.8}).
\end{proof}

The Newton-Maclaurin inequality is as follows,

\begin{proposition}\label{pro4}
  For  $\lambda\in\Gamma_{k}$ and $k>l\geq0$, we have
  \begin{eqnarray}\label{2.7}
    \big(\frac{S_{k}(\lambda)}{C_n^k}\big)^{\frac{1}{k}}\leq\big(\frac{S_{l}(\lambda)}{C_n^l}\big)^{\frac{1}{l}},
  \end{eqnarray}
  where $C_n^k=\frac{n!}{k!(n-k)!}$. Furthermore we have
  \begin{eqnarray}\label{2.10}
  \sum_{i=1}^{n}\frac{\partial S_{k}^{\frac{1}{k}}}{\partial\lambda_{i}}\geq[C_n^k]^{\frac{1}{k}}.
  \end{eqnarray}
\end{proposition}
\begin{proof}
  See \cite{s} for a proof of (\ref{2.7}). For (\ref{2.10}),
   we use (\ref{2.7}) and Proposition \ref{pro1} to get
  \begin{eqnarray}
    \sum_{i=1}^n\frac{\partial S_k^{\frac{1}{k}}(\lambda)}{\partial\lambda_i}=\frac{1}{k}S_k^{\frac{1}{k}-1}\sum^n_{i=1}S_{k-1}(\lambda|i)=
    \frac{n-k+1}{k}S_{k}^{\frac{1}{k}-1}S_{k-1}(\lambda)\geq[C_n^k]^{\frac{1}{k}}.\nonumber
  \end{eqnarray}
\end{proof}

Then we give some useful inequalities of elementary symmetric functions.

\begin{proposition}\label{pro6}
  Suppose $\lambda=(\lambda_{1}, \cdots, \lambda_{n})\in\Gamma_k$, $k\geq1$, satisfies $\lambda_1<0$. Then we have
  \begin{eqnarray}\label{2.5}
    \frac{\partial S_{k}(\lambda)}{\partial\lambda_1}\geq\frac{1}{n-k+1}\sum_{i=1}^{n}\frac{\partial S_k}{\partial \lambda_i}.
  \end{eqnarray}
  and
  \begin{eqnarray}\label{2.11}
    \sum_{i=1}^{n}\frac{\partial S_k(\lambda)}{\partial \lambda_i}\geq(-\lambda_1)^{k-1},\quad\forall 1\leq k\leq n.
  \end{eqnarray}
  \begin{proof}
    See Lemma 3.9 in \cite{c} for the proof of (\ref{2.5}), and \cite{cz} or \cite{cw} for (\ref{2.11}).
  \end{proof}
  \end{proposition}
The following  proposition is useful to establishments of
gradient estimates(for $f=f(x,u,Du)$) and double normal estimates(for $m\leq \frac{n}{2}$).
This proposition also indicates the  major difference between our cases$(m\geq2)$ and the $k$-Hessian$(m=1)$.
\begin{proposition}\label{pro7}
  Let $\mu=(\mu_1,\cdots,\mu_n)$ with $\mu_1\geq\cdots\geq\mu_n$,
  $\lambda=\{\mu_{i_1}+\mu_{i_2}+\cdots+\mu_{i_m}|  1\leq i_1< i_2<\cdots<i_m\leq n\}$
   and $2\leq k\leq  \frac{n-m}{n}C_n^m$.
   If $\mu\in\Gamma^{(m)}_k$ and $\mu_n<-\delta L<0$, where $\delta$ is a small positive constant,  then there exits a constant $\theta_1=(\frac{\delta^{k}}{(C_n^m)!4^k})^{k-1}$ such that
  \begin{eqnarray}
    \sum_{i=1}^{C_n^m}\frac{\partial S_k(\lambda)}{\partial \lambda_i}\geq \theta_1  L^{k-1}.\label{p71}
  \end{eqnarray}
  Furthermore, if in addition that $-\delta_1L\leq\lambda_i\leq mL$, $\forall 1\leq i\leq C_n^m$, with $\delta_1=\frac{\delta^{k}}{(C_n^m)!4^k}$,
  then there exists a constant $\theta_2=\frac{\delta^{k-1}}{2^km^{k-1}(C_n^m)^3}$, such that, for $1\leq j\leq C_n^m$
  \begin{eqnarray}
    \frac{\partial S_k(\lambda)}{\partial\lambda_{i}}\geq
    \theta_2 \sum_{j=1}^{C_n^m}\frac{\partial S_k(\lambda)}{\partial \lambda_j}.\label{p72}
  \end{eqnarray}
\end{proposition}
\begin{proof}
  Let $\lambda_1\geq\cdots\geq\lambda_{C_n^m}$.
   We consider the following two cases.

  \textbf{Case1.} $\lambda_{C_n^m}<-\delta_1L$, where $\delta_1=\frac{\delta^{k}}{(C_n^m)!4^k}$.

  It is exactly the case in Proposition \ref{pro6}, so we have
  \begin{eqnarray}
    \sum_{i=1}^{C_n^m}\frac{\partial S_k(\lambda)}{\partial \lambda_i}\geq(\delta_1L)^{k-1}.\label{p72.1}
  \end{eqnarray}

  \textbf{Case2.} $\lambda_{C_n^m}\geq-\delta_1L$.


  We  see that
  \begin{eqnarray}
    \lambda_{C_{n}^{m}}=\sum_{i=n-m+1}^{n-1}\mu_{i}+\mu_{n}\geq-\delta_{1}L.\nonumber
  \end{eqnarray}
  Since $\mu_{n}<-\delta L$ and $\delta_{1}<\frac{\delta}{2}$, we obtain
  \begin{eqnarray}
 \sum_{i=n-m+1}^{n-1}\mu_{i}\geq\frac{\delta}{2}L,\quad \ \mu_{n-m+1}>0.\nonumber
  \end{eqnarray}
  It follows that
  \begin{eqnarray}\label{4.58}
    \lambda_{C_{n}^{m}-C_{n-1}^{m-1}}\geq\sum_{i=n-m+1}^{n-1}\mu_{i}+\mu_{n-m}>\frac{\delta}{2}L.
  \end{eqnarray}

  Now we can write
  \begin{eqnarray}
    \lambda_1\geq\cdots\geq\lambda_{p}\geq\frac{\delta}{2}L\geq\lambda_{p+1}\geq\cdots\geq\lambda_{q}>0\geq\lambda_{q+1}
    \geq\cdots\lambda_{C_{n}^{m}}\geq-\delta_{1}L.\nonumber
  \end{eqnarray}
Denote $\lambda'=(\lambda_1,\cdots, \lambda_{p})$, $\lambda''=(\lambda_{1},\cdots,\lambda_{q})$, and $\lambda'''=(\lambda_{q+1},\cdots,\lambda_{C_{n}^{m}})$. We pint out that $\lambda'''$ may be empty. From (\ref{4.58}) we see that
\begin{eqnarray}
  p\geq C_{n}^{m}-C_{n-1}^{m-1}\geq k,\nonumber
\end{eqnarray}
and, use $\lambda_1\leq mL$ (only for the second inequality of (\ref{4.60})) to get
\begin{eqnarray}\label{4.60}
  C_{p}^{k-1}(\frac{\delta}{2})^{k-1}L^{k-1}\leq S_{k-1}(\lambda')\leq C_{p}^{k-1}m^{k-1}L^{k-1}.
\end{eqnarray}
We also have
\begin{eqnarray}
  S_{k-1}(\lambda')\leq S_{k-1}(\lambda'')\leq(C_{n}^{m}-1)S_{k-1}(\lambda'),\label{4.61}
\end{eqnarray}
since every element of $\lambda''$ is positive.

By Proposition \ref{pro2} and (\ref{2.0}), we have
\begin{eqnarray}
  \sum_{i=1}^{C_n^m}\frac{\partial S_k(\lambda)}{\partial \lambda_i}&=&\sum_{i=1}^{C_{n}^{m}}S_{k-1}(\lambda|i)=(C_{n}^{m}-k+1)S_{k-1}(\lambda)\nonumber\\
  &=&(C_{n}^{m}-k+1)[S_{k-1}(\lambda'')+\sum_{i=1}^{k-1}C_k^iS_{k-1,i}(\lambda'',\lambda''')],\label{4.62a}
\end{eqnarray}
where $S_{k-1,i}(\lambda'',\lambda''')$ is the mixed symmetric function. Recall $\delta_{1}=\frac{\delta^{k}}{(C_n^m)!4^k}$ and (\ref{4.60}), such that
\begin{eqnarray}
  |\sum_{i=1}^{k-1}C_k^iS_{k-1,i}(\lambda'',\lambda''')|\leq (\frac{\delta}{2})^{k}L^{k-1}\leq\frac{1}{2} S_{k-1}(\lambda').\label{4.62b}
\end{eqnarray}
Plug (\ref{4.61}) and (\ref{4.62b}) into (\ref{4.62a}),
\begin{eqnarray}\label{4.63}
  \frac{(C_{n}^{m}-k+1)}{2}S_{k-1}(\lambda')\leq\sum_{i=1}^{C_n^m}\frac{\partial S_k(\lambda)}{\partial \lambda_i}\leq C_{n}^{m}(C_{n}^{m}-k+1)S_{k-1}(\lambda').
\end{eqnarray}
 Note that we don't need $\lambda_i\leq mL$ in the first inequality. Combining (\ref{p72.1}), (\ref{4.60}) and (\ref{4.63}), we prove the (\ref{p71}).

We also have
\begin{eqnarray}
  S_{k-1}(\lambda|1)
  &\geq&S_{k-1}(\lambda'|1)+\sum_{i=1}^{k-1}C_k^iS_{k-1,i}(\lambda''|1,\lambda''').
\end{eqnarray}
Due to $p\geq k$, $\delta_{1}=\frac{\delta^{k-1}}{(C_n^m)!4^k}$, (\ref{4.60}) and (\ref{4.63}), we have
\begin{eqnarray}
 S_{k-1}(\lambda|1)&\geq&\frac{1}{2}S_{k-1}(\lambda'|1)\geq\frac{\delta^{k-1}}{2^km^{k-1}C_n^m}S_{k-1}(\lambda')\nonumber\\
  &\geq&\frac{\delta^{k-1}}{2^km^{k-1}(C_n^m)^3}\sum_{i=1}^{C_n^m}\frac{\partial S_k(\lambda)}{\partial \lambda_i}.\nonumber
\end{eqnarray}
Then we proved the (\ref{p72}) since $S_{k-1}(\lambda|i)\geq S_{k-1}(\lambda|1)$ for $1\leq i\leq C_n^m$.
\end{proof}

%
  Finally, we give a key inequality which play an important role in the establishment of the
   double normal derivative estimate(see Theorem \ref{th5.2}).

  \begin{proposition}\label{pro5}
    Suppose $\lambda=(\lambda_1, \cdots, \lambda_n)\in\Gamma_k$, $k\geq2$, and $\lambda_2\geq\cdots\geq\lambda_n$. If $\lambda_1>0$, $\lambda_1\geq\delta\lambda_2$, and $\lambda_n\leq-\varepsilon\lambda_1$ for small positive constants $\delta$ and $\varepsilon$, then
    we have
    \begin{eqnarray}\label{2.13}
      S_{l}(\lambda|1)\geq c_0 S_l(\lambda),\quad\forall l=0, 1,\cdots, k-1,
    \end{eqnarray}
    where $c_0=\min\{\frac{\varepsilon^2\delta^2}{2(n-2)(n-1)}, \frac{\varepsilon^2\delta}{4(n-1)}\}$.
  \end{proposition}
One can find a generalized inequality and the proof in \cite{cz}. For completeness we give a proof for our case as same as in \cite{mq}.
\begin{proof}
 For $l= 0$, (\ref{2.13}) holds directly. In the following, we assume $1 \leq l \leq k-1$.

 Firstly, if $\lambda_1\geq\lambda_2$, we have from (\ref{2.3})
 \begin{eqnarray}\label{2.15}
   \lambda_1 S_{l-1}(\lambda|1n)\geq\frac{l}{n-1}S_l(\lambda|n).
 \end{eqnarray}
 If $\lambda_1<\lambda_2$, use $\lambda_1\geq\delta\lambda_2$ and (\ref{2.1})to get
 \begin{eqnarray}\label{2.16}
   \lambda_1S_{l-1}(\lambda|1n)\geq\delta\lambda_2S_{l-1}(\lambda|2n)\geq\delta\frac{l}{n-1}S_l(\lambda|n).
 \end{eqnarray}
 It follows from (\ref{2.15}) and (\ref{2.16}) that
 \begin{eqnarray}\label{2.17}
   (-\lambda_n)S_{l-1}(\lambda|1n)\geq\delta\varepsilon\frac{l}{n-1}S_l(\lambda|n)\geq\delta\varepsilon\frac{l}{n-1}S_l(\lambda).
 \end{eqnarray}
 We use $S_l(\lambda)=S_l(\lambda|n)+\lambda_n S_{l-1}(\lambda|n)\leq S_l(\lambda|n)$, for $\lambda_n<0$, in the second inequality. Then we consider the following two cases.

 \textbf{Case1.} $S_l(\lambda|1)\geq \theta(-\lambda_n)S_{l-1}(\lambda|1n)$, $\theta$ is a small positive number to be determined.

Use (\ref{2.17}) directly to obtain
\begin{eqnarray}\label{2.18}
  S_l(\lambda|1)\geq\theta\delta\varepsilon\frac{l}{n-1}S_l(\lambda).
\end{eqnarray}

\textbf{Case2.} $S_l(\lambda|1)< \theta(-\lambda_n)S_{l-1}(\lambda|1n)$.

From proposition \ref{pro1} we have
\begin{eqnarray}
  (l+1)S_{l+1}(\lambda|1)&=&\sum_{i=2}^{n}\lambda_i S_l(\lambda|1i)=\sum_{i=2}^{n}\lambda_i\big[S_l(\lambda|1)-\lambda_i S_{l-1}(\lambda|1i)\big]\nonumber\\
  &=&\sum_{i=}^{n}\lambda_i S_l(\lambda|1)-\sum_{i=2}^{n}\lambda^2_i S_{l-1}(\lambda|1i)\nonumber\\
  &\leq&(n-2)\lambda_2S_l(\lambda|1)-\lambda_n^2S_{l-1}(\lambda|1n)\nonumber\\
  &\leq&(\frac{n-2}{\delta}\theta-\varepsilon)\lambda_1(-\lambda_n)S_{l-1}(\lambda|1n)
  =-\frac{\varepsilon}{2}\lambda_1(-\lambda_n)S_{l-1}(\lambda|n),
\end{eqnarray}
if we choose $\theta=\frac{\varepsilon\delta}{2(n-2)}$ in the last equality. From (\ref{2.17}), we have
\begin{eqnarray}
  S_{l+1}(\lambda|1)\leq-\frac{\varepsilon^2\delta m}{2(n-1)(l+1)}\lambda_1S_l(\lambda),
\end{eqnarray}
then
\begin{eqnarray}
S_l(\lambda|1)&=&\frac{S_{l+1}(\lambda)-S_{l+1}(\lambda|1)}{\lambda_1}\geq-\frac{S_{l+1}(\lambda|1)}{\lambda_1}\nonumber\\
&\geq&\frac{\varepsilon^2\delta}{2(l+1)}\frac{l}{n-1}S_l(\lambda)>\frac{\varepsilon^2\delta}{4(n-1)}S_l(\lambda).
\end{eqnarray}
Hence (\ref{2.13}) holds.
\end{proof}

\section{$C^{0}$ Estimate}\label{sec3}
Following the idea of Lions-Trudinger-Urbas \cite{ltu}, we prove the following theorem.
\begin{theorem}\label{th1}
  Let $\Omega\subset \mathbb{R}^{n}\ (n\geq3)$ be a bounded domain with $C^{1}$ boundary,
  and $\nu$ be the unit outer normal vector of $\partial \Omega$.
  Suppose that $u\in C^{2}(\overline{\Omega})\cap C^{3}(\Omega)$
  is an $k$-admissible solution of the following Neumann boundary problem,
  \begin{equation}\label{eq3}
   \left\{
    \begin{aligned}
    S_{k}(W)&=f(x),\quad\text{in}\ \ \Omega,\nonumber\\
     u_{\nu}&=-a(x)u+b(x),\quad\text{on}\ \ \partial\Omega.
    \end{aligned}
    \right.
\end{equation}
where $f>0$ and $a, b\in C^{3}(\partial\Omega)$ with $\inf\limits_{\partial \Omega}a(x)>\sigma$. Then
\begin{eqnarray}
  \sup_{\overline{\Omega}}|u|\leq \frac{C}{\sigma}
\end{eqnarray}
where $C$ depends on $k$, $n$, $a$, $b$, $f$ and $diam(\Omega)$.
\end{theorem}
\begin{proof}
 Because $f>0$, the comparison principle tells us that $u$ attains its maximum on the boundary. At the maximum point $x_{0}\in\partial\Omega$ we have
 \begin{eqnarray}
   0\leq u_{\nu}(x_{0})=(-au+b)(x_{0}).\nonumber
 \end{eqnarray}
 It implies that
 \begin{eqnarray}
   u(x)\leq u(x_{0})\leq\frac{\sup\limits_{\partial\Omega} b}{\inf\limits_{\partial\Omega}a}.
 \end{eqnarray}

 Assume $0\in\Omega$ and let $w=u-A|x|^{2}$. We obtain
 \begin{eqnarray}
   F[A|x|^{2}]\geq f=F[u],\nonumber
 \end{eqnarray}
 if we choose $A$ large enough depends on k, n and $\sup f$. Similarly $w$ attains its minimum on the boundary by comparison principle. At the minimum point $x_{1}\in\partial\Omega$ we have
 \begin{eqnarray}
   0\geq w_{\nu}(x_{1})=(-au+b)(x_{1})-2Ax_{0}\cdot\nu.\nonumber
 \end{eqnarray}
 We use $w(x)\geq w_(x_{1})$ to get
 \begin{eqnarray}
   u(x)\geq-\frac{|\inf\limits_{\partial\Omega} b-2AL(L+1)|}{\sup\limits_{\partial\Omega} a}\geq-\frac{|\inf\limits_{\partial\Omega} b-2AL(L+1)|}{\inf\limits_{\partial\Omega} a},
 \end{eqnarray}
 where $L=diam(\Omega)$. Then we complete the proof of Theorem \ref{th1}.
\end{proof}

\section{Global gradient estimate}\label{sec4}

Throughout the rest of this paper, we always admit the Einstein's summation convention. All repeated indices come from 1 to n.
We will denote $F(D^2u)=S_k(W)$ and
\begin{eqnarray*}
  F^{ij}=\frac{\partial F(D^2u)}{\partial u_{ij}}=\frac{\partial S_{k}(W)}{\partial w_{\overline{\alpha}\overline{\beta}}}
\frac{\partial w_{\overline{\alpha}\overline{\beta}}}{\partial u_{ij}}.
\end{eqnarray*}
From (\ref{w1}) and (\ref{sw1}) we have, for any $1\leq j\leq n$,
\begin{eqnarray}
F^{ii}=\sum_{i\in\overline{\alpha}}\frac{\partial S_k(W)}{\partial w_{\overline{\alpha}\overline{\alpha}}}.
\end{eqnarray}
Throughout the rest of the paper, we will denote $\mathcal{F}=\sum\limits_{i=1}^{n}F^{ii}=m\sum\limits_{N_{\overline{\alpha}}=1}^{C_n^m}S_{k-1}(W|N_{\overline{\alpha}})$ for simplicity.

\subsection{Interior gradient estimate}
Chou-Wang \cite{cw} gave the interior gradient estimates for $k$-Hessian equations. In a similar way, we will prove the following theorem.
\begin{theorem}\label{th2}
  Let $\Omega\subset \mathbb{R}^n\ (n\geq3)$ be a bounded domain and $2\leq k\leq  \frac{n-m}{n}C_n^m$. Suppose that $u\in C^{3}(\Omega)$
  is a k-admissible solution of the following equation,
  \begin{equation}\label{eq4}
    S_{k}(W)=f(x,u,Du),\quad\text{in}\ \ \Omega,
\end{equation}
where  $f(x,z,p)\in C^1(\overline{\Omega}\times [-M_0,M_0]\times\mathbb{R}^n)$ is a nonnegative function,  $M_0=\sup_{\overline{\Omega}}|u|$. We also assume that
\begin{eqnarray}
  |f|_{C^0}+\sum_{i=1}^{n}|f_{x_i}|_{C^0}+|f_z|_{C^0}+\sum_{i=1}^{n}|f_{p_i}|_{C^0}|Du|_{C^0}\leq L_1(1+|Du|_{C^0}^{2k-1}),\label{4.3}
\end{eqnarray}
for some constant $L_1$ independent of $|Du|_{C^0}$.
For any $B_r(y)\subset \Omega$, we have
\begin{eqnarray}
  \sup_{B_{\frac{r}{2}}(y)}|Du|\leq C_1+C_2\frac{M_0}{r},\label{4.4}
\end{eqnarray}
where $C_{1}$ depends only on $M_0$, $L_1$, $n$, $m$, and $k$, and $C_{2}$ depends only on $L_1$, $n$, $m$, and $k$. Moreover, if $f\equiv \text{constant}$, then $C_1=0$.
\end{theorem}
\begin{proof}
Assume $y=0\in\Omega$ and $B_r(0)\subset\Omega$. Choose the auxiliary function as
\begin{eqnarray}
G(x)=\rho (x)\varphi(u)|Du|^{2},
\end{eqnarray}
where $\rho(x)=(1-\frac{x^{2}}{r^{2}})^{2}$ such that $|D\rho|\leq b_0\rho^{\frac{1}{2}}$ and $|\nabla^{2}\rho| \leq b_0^2$, with $b_0=\frac{4}{r}$, and $\varphi(u)=(M-u)^{-\frac{1}{2}}$ with $M=4M_0$. It is easy to see that
\begin{eqnarray}\label{2.4}
\varphi''-\frac{2(\varphi')^{2}}{\varphi} \geq \frac{1}{16}M^{-\frac{5}{2}}.
\end{eqnarray}

Suppose $G$ attains its maximum at the point $x_{0}\in\Omega=B_{r}(0)$. In the following, all the calculations are at $x_{0}$. First, we have
\begin{eqnarray}
0=G_{i}(x_{0})=\rho_{i}\varphi|Du|^{2}+\rho u_{i}\varphi'|Du|^{2}+2\rho\varphi u_{k}u_{ki},\quad i=1,\cdot\cdot\cdot,n.\nonumber
\end{eqnarray}
After a rotation of the coordinates, we may assume that the matrix $(u_{ij})_{n\times n}$ is diagonal at $x_0$, so are $W$ and $(F^{ij})_{n\times n}$. The above identity can be rewrote as
\begin{eqnarray}\label{2.6}
  u_{i}u_{ii}=-\frac{1}{2\rho\varphi}(\varphi\rho_{i}+\rho\varphi'u_{i})|Du|^{2},\quad
  i=1,\cdot\cdot\cdot,n.
\end{eqnarray}
We also have
\begin{eqnarray}
  G_{ij}(x_{0})&=&2\rho\varphi u_{k}u_{kij}+2\rho\varphi u_{ki}u_{kj}+2\rho\varphi'(u_{i}u_{k}u_{kj}+u_{j}u_{k}u_{ki})\\ \nonumber
  &&+2\varphi(\rho_{i}u_{k}u_{kj}+\rho_{j}u_{k}u_{ki})+\rho u_{ij}\varphi'|Du|^{2}+\rho\varphi''|Du|^{2}u_{i}u_{j}\\ \nonumber
  &&+\varphi''|Du|^{2}(\rho_{i}u_{j}+\rho_{j}u_{i})+\rho_{ij}\varphi|Du|^{2}.\nonumber
\end{eqnarray}
Use the maximum principle to get
\begin{eqnarray}
  0&\geq& F^{ij}G_{ij}=F^{ii}G_{ii}\\
&=& 2\rho\varphi u_{k}F^{ii}u_{iik}+2\rho\varphi F^{ii}u_{ii}^{2}+4\rho\varphi'F^{ii}u_{i}^{2}u_{ii}+4\varphi F^{ii}\rho_{i}u_{i}u_{ii}\nonumber\\
 &&+ \rho\varphi'|Du|^{2}F^{ii}u_{ii}+\rho\varphi''|Du|^{2}F^{ii}u_{i}^{2}+2\varphi'|Du|^{2}F^{ii}\rho_{i}u_{i}
  +F^{ii}\rho_{ii}\varphi|Du|^{2}.\nonumber
\end{eqnarray}
From the facts that
\begin{eqnarray}\label{2.9}
  F^{ii}u_{ii}=kf,\quad \quad F^{ii}u_{iil}=f_{x_l}+f_{z}u_{l}+f_{p_l}u_{ll},
\end{eqnarray}
we have
\begin{eqnarray}
  0&\geq& 2\rho\varphi u_{l}(f_{l}+f_{z}u_{l})+2\rho\varphi f_{p_l}u_{l}u_{ll}+2\rho\varphi F^{ii}u_{ii}^{2}\nonumber\\
 &&+4\rho\varphi'F^{ii}u_{i}^{2}u_{ii}+4\varphi F^{ii}\rho_{i}u_{i}u_{ii}+ mf\rho\varphi'|Du|^{2}+\rho\varphi''|Du|^{2}F^{ii}u_{i}^{2}\nonumber\\
 &&+2\varphi'|Du|^{2}F^{ii}\rho_{i}u_{i}+F^{ii}\rho_{ii}\varphi|Du|^{2}.\nonumber
\end{eqnarray}
Assume $|Du|(x_0)\geq b_0$, otherwise we have (\ref{4.4}). By (\ref{4.3}) and (\ref{2.6}), which used to deal with the second, fourth and fifth terms, then
\begin{eqnarray}
  0&\geq&-4L_1(\varphi+\varphi')|Du|^{2k+1}+2\rho\varphi F^{ii}u_{ii}^{2}-2\varphi'|Du|^{2}F^{ii}u_{i}\rho_{i}-\frac{2\varphi|Du|^{2}}{\rho}F^{ii}\rho_{i}^{2}\nonumber \\
  &&+(\varphi''-\frac{2\varphi'^{2}}{\varphi})\rho |Du|^{2}F^{ii}u_{i}^{2}+\varphi|Du|^{2}F^{ii}\rho_{ii}.\nonumber
\end{eqnarray}
By (\ref{2.4}) and properties of $\rho$ we have
\begin{eqnarray}\label{2.12}
  0&\geq&2\rho\varphi F^{ii}u_{ii}^{2} -2b_{0}\varphi'\rho^{\frac{1}{2}}|Du|^{3}\mathcal{F}
  -3b_{0}^2\varphi|Du|^{2}\mathcal{F}\nonumber\\
  && -4L_1(\varphi+\varphi')|Du|^{2k+1}.
\end{eqnarray}

Assume $G(x_{0})\geq 20nb_{0}^{2}M^{\frac{3}{2}}$, otherwise we have (\ref{4.4}), which implies that $|Du|\geq \frac{2\sqrt{5n}b_{0}M^{\frac{3}{4}}}{\rho^{\frac{1}{2}}\varphi^{\frac{1}{2}}}$ at $x_{0}$. There exists at least one index $i_{0}$ such that $|u_{i_{0}}| \geq \frac{|Du|}{\sqrt{n}}$. By (\ref{2.6}), it is not hard to get
\begin{eqnarray}
  u_{i_{0}i_{0}}&=&-(\frac{\varphi'}{2\varphi}+\frac{\rho_{i_{0}}}{2\rho u_{i_{0}}})|Du|^{2} \nonumber\\
  &\leq& -(\frac{\varphi'}{2\varphi}-\frac{1}{20M})|Du|^{2}  \nonumber\\
   &\leq& -\frac{\varphi'}{4\varphi}|Du|^{2}.\label{4.11}
\end{eqnarray}

Let $u_{11}\geq\cdot\cdot\cdot\geq u_{nn}$, from (\ref{2.1}) and (\ref{4.11}) we have
\begin{eqnarray}\label{2.14}
  u_{nn}\leq-\frac{\varphi'|Du|^{2}}{4\varphi},
  \quad F^{11}\leq\cdot\cdot\cdot\leq F^{nn}.
  \label{4.12}
\end{eqnarray}
The second part implies that $F^{nn}\geq\frac{1}{n}\mathcal{F}$.
Returning to (\ref{2.12}) we have
\begin{eqnarray}
  0&\geq&2\rho\varphi F^{nn}u_{nn}^{2} -2b_{0}\varphi'\rho^{\frac{1}{2}}|Du|^{3}\mathcal{F}
  -3b^2_{0}\varphi|Du|^{2}\mathcal{F}-4L_1(\varphi+\varphi')|Du|^{2k+1}\nonumber
  \\
  &\geq& \frac{\rho\varphi'^{2}}{8n\varphi}|Du|^{4}\mathcal{F}-2b_{0}\varphi'\rho^{\frac{1}{2}}|Du|^{3}\mathcal{F}
  -3b^2_{0}\varphi|Du|^{2}\mathcal{F}-4L_1(\varphi+\varphi')|Du|^{2k+1}.\label{4.13}
\end{eqnarray}
Both sides of (\ref{4.13}) multiplied by $\rho\varphi^{3}$, then we have
\begin{eqnarray}
 0&\geq& (\frac{2G^{2}}{125nM^{3}}-\frac{8b_{0}G^{\frac{3}{2}}}{3M^{\frac{9}{4}}}-\frac{6b^2_{0}G}{M^{\frac{3}{2}}})\mathcal{F}
 -4L_1(\frac{1}{M^{\frac{5}{4}}}+\frac{1}{M^{\frac{9}{4}}})|Du|^{2k-2}G^{\frac{3}{2}}.\label{4.14}
\end{eqnarray}
By (\ref{4.12}), we can choose $\delta=\frac{\varphi'}{4\varphi}$, $L=|Du|^2$ and $\theta_1=(\frac{(\varphi')^{k}}{(C_n^m)!(16\varphi)^{k}})^{k-1}$ in the Proposition \ref{pro7}, such that
\begin{eqnarray}
  \mathcal{F}\geq \theta_1|Du|^{2k-2}.\nonumber
\end{eqnarray}
Then,
\begin{eqnarray}
  0\geq \frac{2G}{125nM^{3}}-\frac{8b_{0}G^{\frac{1}{2}}}{3M^{\frac{9}{4}}}-\frac{6b^2_{0}}{M^{\frac{3}{2}}}
  -4\theta_1^{-1}L_1(\frac{1}{M^{\frac{5}{4}}}+\frac{1}{M^{\frac{9}{4}}})G^{\frac{1}{2}}.\nonumber
\end{eqnarray}
It follows that
\begin{eqnarray}
  G^{\frac{1}{2}}(x_0)\leq C_1+C_2\frac{M^{\frac{3}{4}}}{r}.\nonumber
\end{eqnarray}
Thus
\begin{eqnarray}
  \sup_{B_{\frac{r}{2}}}|Du|\leq C_1+C_2\frac{M}{r},
\end{eqnarray}
where $C_{1}$ depends only on $M$, $L_1$, $n$, $m$, and $k$, and $C_{2}$ depends only on $L_1$, $n$, $m$, and $k$.
It is not hard to see that $C_1=0$ when $f\equiv \text{constant}$.
\end{proof}

In fact, if we only consider for $f=f(x,u)>0$ in the equation (\ref{eq4}), we could remove the restriction to $k$ in Theorem \ref{th2}
 and the following Theorem \ref{th4.1}. Precisely, we have
\begin{theorem}\label{th3}
  Let $\Omega\subset \mathbb{R}^n$ be a bounded domain and $2\leq k\leq C_n^m$. Suppose that $u\in C^{3}(\Omega)$
  is a k-admissible solution of the following equation,
  \begin{equation}\label{eq7}
    S_{k}(W)=f(x,u),\quad\text{in}\ \ \Omega,
\end{equation}
where  $f(x,z)\in C^1(\overline{\Omega}\times [-M_0,M_0]\times\mathbb{R}^n)$ is a positive function,  $M_0=\sup_{\overline{\Omega}}|u|$. We also assume that
\begin{eqnarray}
  |f|_{C^1(\overline{\Omega}\times [-M_0,M_0]\times\mathbb{R}^n)}\leq L_1,\label{4.5}
\end{eqnarray}
for some constant $L_1$.
For any $B_r(y)\subset \Omega$, we have
\begin{eqnarray}
  \sup_{B_{\frac{r}{2}}(y)}|Du|\leq C_1+C_2\frac{M_0}{r},\label{4.6}
\end{eqnarray}
where $C_{1}$ depends only on $M_0$, $L_1$, $\min{f}$, $n$, $m$, and $k$, and $C_{2}$ depends only on $L_1$, $\min{f}$, $n$, $m$, and $k$. Moreover, if $f\equiv \text{constant}$, then $C_1=0$.
\end{theorem}
\begin{proof}
The proof of this result is essentially the same as the proof of Theorem \ref{th2}, the only difference being that we cannot apply Proposition \ref{pro7} to give a lower bound to $\mathcal{F}$. Instead, we use the Newton-Maclaurin inequality. From (\ref{4.11}) we still have
\begin{eqnarray}
  u_{nn}\leq-\frac{\varphi'|Du|^{2}}{4\varphi},\quad F^{11}\leq\cdot\cdot\cdot\leq F^{nn}.
\end{eqnarray}
The second part implies that
\begin{eqnarray}
  F^{nn}&\geq&\frac{1}{n}\mathcal{F}=\frac{m}{n}\sum_{i=1}^{C_n^m}S_{k-1}(\lambda|i)\nonumber\\
  &=&\frac{m(C_n^m-k+1)}{n}S_{k-1}(\lambda).\nonumber
\end{eqnarray}
By the Newton-Maclaurin inequality, we have
\begin{eqnarray}
  F^{nn}\geq \frac{1}{n}\mathcal{F}\geq c S_k^{\frac{1}{k}}(\lambda)\geq c(\min{f})^{\frac{1}{k}},\label{4.18}
\end{eqnarray}
where $c=c(n,m,k)$ a universal constant. It is not hard to see, a different version of (\ref{4.14}), that
\begin{eqnarray}
 0&\geq& (\frac{2G^{2}}{125nM^{3}}-\frac{8b_{0}G^{\frac{3}{2}}}{3M^{\frac{9}{4}}}-\frac{6b^2_{0}G}{M^{\frac{3}{2}}})\mathcal{F}
 -4L_1(\frac{1}{M^{\frac{5}{4}}}+\frac{1}{M^{\frac{9}{4}}})G.\label{4.15}
\end{eqnarray}
Plug (\ref{4.18}) into(\ref{4.15}), then
\begin{eqnarray}
  0\geq \frac{2G}{125nM^{3}}-\frac{8b_{0}G^{\frac{1}{2}}}{3M^{\frac{9}{4}}}-\frac{6b^2_{0}}{M^{\frac{3}{2}}}
  -4c^{-1}(\min{f})^{-\frac{1}{k}}L_1(\frac{1}{M^{\frac{5}{4}}}+\frac{1}{M^{\frac{9}{4}}}).\nonumber
\end{eqnarray}
Thus we have
\begin{eqnarray}
  \sup_{B_{\frac{r}{2}}}|Du|\leq C_1+C_2\frac{M}{r},
\end{eqnarray}
where $C_{1}$ depends only on $M_0$, $L_1$, $\min{f}$, $n$, $m$, and $k$, and $C_{2}$ depends only on $L_1$, $\min{f}$, $n$, $m$, and $k$. It is not hard to see that $C_1=0$ when $f\equiv \text{constant}$.
\end{proof}

\subsection{Gradient estimate near boundary}
In this subsection, we will establish a gradient estimate in the small neighborhood near
boundary. We use a similar method as in Ma-Qiu \cite{mq} with minor changes. We define
\begin{eqnarray}
  d(x)=dist(x,\partial\Omega),\nonumber\\
  \Omega_{\mu}=\{x\in\Omega|\ d(x)<\mu\}.
\end{eqnarray}
 It is well known that there exists a small positive universal
constant $\mu_0$ such that $d(x)\in C^k(\Omega_{\mu}),\ \forall0<\mu\leq\mu_0$, provided $\partial\Omega\in C^k$.  As in Simon-Spruck \cite{ss} or Lieberman \cite{l}
(in page 331), we can extend $\nu$ by $\nu=-D d$ in $\Omega_{\mu}$ and note that $\nu$ is a $C^2(\overline{\Omega_{\mu}})$ vector field. As mentioned in the book \cite{l}, we also have the following formulas
\begin{eqnarray}
  |D\nu|+|D^2\nu|\leq C(n,\Omega),\quad\text{in}\ \Omega_{\mu},\nonumber\\
  \sum\limits_{i=1}^n\nu^iD_j\nu^i=\sum\limits_{i=1}^n\nu^iD_i\nu^j=\sum\limits_{i=1}^nd_id_{ij}=0,\ |\nu|=|Dd|=1,\quad\text{in}\ \Omega_{\mu}.\label{d}
\end{eqnarray}

\begin{theorem}\label{th4.1}
Suppose $\Omega\subset\mathbb{R}^n\ (n\geq3)$ is a bounded domain with $C^ 3$ boundary, and $2\leq k\leq  \frac{n-m}{n}C_n^m$,
Let $f(x,z,p)\in C^1(\overline{\Omega}\times [-M_0,M_0]\times\mathbb{R}^n)$ is a nonnegative function and
$\phi\in C^3(\overline{\Omega}\times [-M_0,M_0])$, $M_0=\sup_{\overline{\Omega}}|u|$.
We also assume that there exists constants $L_1$ (independent of $|Du|_{C^0}$) and $L_2$ such that
\begin{eqnarray}
|f|_{C^0}+\sum_{i=1}^{n}|f_{x_i}|_{C^0}+|f_z|_{C^0}+\sum_{i=1}^{n}|f_{p_i}|_{C^0}|Du|_{C^0}&\leq& L_1(1+|Du|_{C^0}^{2k-1}),\label{4.20}\\
|\phi|_{C^3(\overline{\Omega}\times [-M_0,M_0])}&\leq& L_2.\label{4.19}
\end{eqnarray}
If $u \in C ^3 (\Omega)\cap C^1 (\overline{\Omega})$ is a $k$-admissible solution of  equation
\begin{equation}
  \left\{
  \begin{aligned}
  S_{k}(W)&=f(x,u,Du),\quad\text{in}\ \Omega,\\
  u_{\nu}&=\phi(x,u),\quad\text{on}\ \partial\Omega.
  \end{aligned}
  \right.
\end{equation}
 Then we have
\begin{eqnarray}
\sup_{\Omega_{\mu}}|Du|\leq C,
\end{eqnarray}
where $C$ is a constant depends only on $n$, $k$, $m$, $\mu$, $M_0$, $L_1$, $L_2$ and $\Omega$.
\end{theorem}
\begin{proof}
Let
\begin{eqnarray}
  G(x):=\log|Dw|^{2}+h(u)+\alpha_{0}d(x),\quad \text{in}\ \Omega_{\mu}\ \forall0<\mu\leq\mu_0
\end{eqnarray}
where
\begin{eqnarray}
  w(x)&=&u(x)+\phi(x,u)d(x),\label{2.19} \\
  h(x)&=&-\frac{1}{2}\log(1+4M_0-u),\quad  h''-2h'^{2}=0,\label{4.25}
\end{eqnarray}
and $\alpha_{0}$ is a constant to be determined.

Above and throughout the text, we always denote $C$  a positive constant depends on some known data.

$\mathbf{Case 1}$: $G$ attains its maximum on the boundary $\partial\Omega$.

If we assume that $|Du|>8nL_{2}$ and $\mu\leq\frac{1}{2L_{2}}$, it follows from (\ref{2.20}) that
\begin{eqnarray}
  \frac{1}{4}|Du|\leq|Dw|\leq2|Du|.
\end{eqnarray}
 Assume $x_{0}$ is the maximum point of $G$, then we have
\begin{eqnarray}\label{3.21}
  0\leq G_{\nu}(x_{0})&=&\frac{D(|Dw|^{2})\cdot\nu}{|Dw|^{2}}+h'u_{\nu}+\alpha_{0}Dd\cdot\nu\nonumber\\
  &=&\frac{D(|Dw|^{2})\cdot\nu}{|Dw|^{2}}+h'\phi-\alpha_{0},
\end{eqnarray}
since $\nu=-Dd$.

On the boundary $\partial \Omega$, by the Neumann condition, we have
\begin{eqnarray}\label{3.22}
  D(|Dw|^{2})\cdot\nu&=&-w_{i}w_{ij}d_{j}\nonumber\\
  &=& -(u_{i}+\phi d_{i})(u_{ij}+D_{ij}\phi d+D_{i}\phi d_{j}+D_{j}\phi d_{i}+\phi d_{ij})d_{j}\nonumber\\
   &=& -(u_{i}+\phi d_{i})(D_{i}(u_{j}d_{j})-u_{j}d_{ij}+D_{i}\phi+D_{j}\phi d_{i}d_{j})\nonumber\\
  &=& (u_{i}+\phi d_{i})(u_{j}d_{ij}-\phi_{z}u_j d_i d_j-\phi_{x_j}d_{i}d_{j})\nonumber\\
  &\leq&  C(|Dw|^{2}+|Dw|).
\end{eqnarray}
where $C=C(|d|_{C^{2}},|\phi|_{C^{1}})$.
Plug (\ref{3.22}) into (\ref{3.21}) to get
\begin{eqnarray}
  0\leq G_{\nu}&\leq&C +\frac{C}{|Dw|}+h'|\phi|-\alpha_{0}\nonumber\\
  &\leq&-C+\frac{C}{|Dw|},
\end{eqnarray}
provided $\alpha_{0}=2C+\frac{2L_{2}}{1+M}+1$. Thus we have $|Dw|(x_{0})\leq 1$ , and $G(x_{0})\leq\alpha_{0}$.

$\mathbf{Case 2}$: $G$ attains its maximum on the interior boundary $\partial\Omega_{\mu}\cap \Omega$.
It follows from the interior gradient estimate (\ref{4.4}) that
\begin{eqnarray}
  \sup_{\partial\Omega_{\mu}\cap\Omega}|Dw|(x_{0})\leq C,
\end{eqnarray}
where $C$  depends only on $M$, $L_1$, $\mu$, $n$, $m$, and $k$. Thus we also have an upper bound for $G(x_{0})$.

$\mathbf{Case 3}$: $G$ attains its maximum at some point $x_{0}\in\Omega_{\mu}$.

We have
\begin{eqnarray}\label{2.20}
  w_{i}&=&(1+\phi_{z}d) u_{i}+R_{i},\\
  R_{i}&=&\phi_{i}d+\phi d_{i},\nonumber
\end{eqnarray}
and the second derivatives
\begin{eqnarray}\label{2.21}
  w_{ij}&=&(1+\phi_{z}d)u_{ij}+ R_{ij},
\end{eqnarray}
with
\begin{eqnarray}
R_{ij}&=&d\phi_{zz}u_{i}u_{j}+(d\phi_{iz}u_{j}+d\phi_{zj}u_{i}+d_{i}\phi_{z}u_{j}+d_{i}\phi_{z}u_{i})\\
    &&+(d\phi_{ij}+d_{i}\phi_{j}+d_{j}\phi_{i}+d_{ij}\phi).\nonumber
\end{eqnarray}
It is easy to see that
\begin{eqnarray}\label{2.23}
  |R_{i}|\leq 2L_{2},\quad|R_{ij}|\leq C(\mu|Du|^{2}+|Du|+1),
\end{eqnarray}
where $C=C(L_{2},n,|d|_{C^3})$. The third derivatives are more complicated,
\begin{eqnarray}\label{2.24}
 w_{ijl}&=&(1+\phi_{z}d)u_{ijl}+d\phi_{zzz}u_{i}u_{j}u_{l}+R_{ijl}\\
    &&+(d\phi_{zz}u_{j}u_{il}+d\phi_{zz}u_{i}u_{jl}+d\phi_{zz}u_{l}u_{ij})  \nonumber\\
    &&+(d\phi_{iz}u_{jl}+d\phi_{jz}u_{il}+d_{i}\phi_{z}u_{jl}+d_{j}\phi_{z}u_{il}+d\phi_{zl}u_{ij}+d_{l}\phi_{z}u_{ij}), \nonumber
\end{eqnarray}
where
\begin{eqnarray}
  R_{ijl}&=&(d\phi_{izz}u_{l}u_{j}+d\phi_{jzz}u_{l}u_{i}+d\phi_{zzl}u_{i}u_{j}+d_{i}\phi_{zz}u_{l}u_{j}+d_{j}\phi_{zz}u_{l}u_{i}
  +d_{l}\phi_{zz}u_{i}u_{j})\nonumber\\
   &&+(d\phi_{ijz}u_{l}+d\phi_{izl}u_{j}+d\phi_{jzl}u_{i}+d_{l}\phi_{zj}u_{i}+d_{l}\phi_{iz}u_{j}+d_{i}\phi_{jz}u_{l}
   +d_{i}\phi_{zl}u_{j}\nonumber\\
   &&+d_{j}\phi_{iz}u_{l}+d_{j}\phi_{zl}u_{i}+d_{ij}\phi_{z}u_{l}+d_{il}\phi_{z}u_{j}+d_{jl}\phi_{z}u_{i})\nonumber\\
   &&+(d\phi_{ijp}+d_{l}\phi_{ij}+d_{j}\phi_{il}+d_{i}\phi_{jl}+d_{ij}\phi_{l}+d_{jl}\phi_{i}+d_{il}\phi_{j}+d_{ijl}\phi).\nonumber
\end{eqnarray}
So we have $|R_{ijl}|\leq C(|Du|^{2}+|Du|+1)$ with $C=C(|d|_{C^{3}}, L_{2})$.

We compute at the maximum point $x_{0}\in\Omega_{\mu}$,
\begin{eqnarray}\label{2.25}
  0=G_{i}(x_{0})=\frac{2w_{l}w_{li}}{|Dw|^{2}}+\alpha_{0}d_{i}+h'u_{i},\quad i=1,\cdot\cdot\cdot,n,
\end{eqnarray}
and
\begin{eqnarray}
  G_{ij}(x_{0})=\frac{2w_{li}w_{lj}}{|Dw|^{2}}+\frac{2w_{l}w_{lij}}{|Dw|^{2}}-\frac{4w_{l}w_{li}w_{q}w_{qj}}{|Dw|^{4}}+\alpha_{0}d_{ij}+h''u_{i}u_{j}
  +h'u_{ij}.\nonumber
\end{eqnarray}
By the maximum principle we have
\begin{eqnarray}\label{2.26}
  0&\geq& F^{ij}G_{ij}=F^{ii}G_{ii}\\
  &=&\frac{2F^{ii}w_{li}^{2}}{|Dw|^{2}}+\frac{2w_{l}F^{ii}w_{iil}}{|Dw|^{2}}-\frac{4F^{ii}(w_{l}w_{li})^{2}}{|Dw|^{4}}+\alpha_{0}F^{ii}d_{ii}\nonumber\\
  &&+h''F^{ii}u_{i}^{2}+h'F^{ii}u_{ii}.\nonumber
\end{eqnarray}
The (\ref{2.25}) implies that $2w_{l}w_{li}=-(\alpha_{0}d_{i}+h'u_{i})|Dw|^{2}$, by the Cauchy-Schwartz inequality, then
\begin{eqnarray}\label{2.27}
  \frac{4F^{ii}(w_{l}w_{li})^{2}}{|Dw|^{4}}&=&\alpha_{0}F^{ii}d_{i}^{2}+2\alpha_{0}h'F^{ii}u_{i}d_{i}+h'^{2}F^{ii}u_{i}^{2}\nonumber\\
  &\leq&2h'^{2}F^{ii}u_{i}^{2}+C\mathcal{F},
\end{eqnarray}
where $C=C(\alpha_{0}, M, n, m, |d|_{C^{3}})$.
Combining (\ref{2.9}), (\ref{4.25}), (\ref{2.27}) with (\ref{2.26}), we get
\begin{eqnarray}\label{2.28}
  0&\geq&\frac{2F^{ii}w_{li}^{2}}{|Dw|^{2}}+\frac{2w_{l}F^{ii}w_{iil}}{|Dw|^{2}}-C\mathcal{F}.
\end{eqnarray}

We may assume that $\mu\leq \frac{1}{2L_{2}}$ and $|Du|(x_{0})\geq 16nL_{2}+1$, so that $\frac{1}{2}\leq1+\phi_{z}d\leq1$ and $\frac{1}{8}
 |Du|^{2}\leq |Dw|^{2}\leq\frac{3}{2}|Du|^{2}$. By (\ref{2.24}), we have
\begin{eqnarray}
  \frac{2w_{l}F^{ii}w_{iil}}{|Dw|^{2}}&=&\frac{1}{|Dw|^{2}}\big(2(1+\phi_{z}d)w_{i}D_{i}f+2d\phi_{zzz}w_{l}u_{l}F^{ii}u_{i}^{2}+
  4d\phi_{zz}F^{ii}u_{ii}u_{i}w_{i}\nonumber\\
  &&+(2d\phi_{zz}w_{l}u_{l}+2d\phi_{zl}w_{l}+\phi_{z}d_{l}w_{l})F^{ii}u_{ii}\nonumber\\
  &&+4d\phi_{iz}F^{ii}u_{ii}w_{i}+4\phi_{z}d_{i}F^{ii}u_{ii}w_{i}+2F^{ii}R_{iil}w_{l}\big)\nonumber\\
  &\geq&-\frac{C}{|Dw|^{2}}\big((\mu |Dw|^{4}+|Dw|^{3}+(\mu+\frac{1}{\mu})|Dw|^{2}+|Dw|)\mathcal{F}+\mu F^{ii}u_{ii}^{2}\big)\nonumber\\
  &&+\frac{2(1+\phi_{z}d)}{|Dw|^{2}}w_{i}D_{i}f,\label{4.40}
\end{eqnarray}
where $C=C(\alpha_{0}, M, n, m, |d|_{C^{3}}, L_{1}, L_{2})$. Here we use the Cauchy inequality and the fact that
$|R_{ijl}|\leq C(|Du|^{2}+|Du|+1)$. Now we deal with the last term. By (\ref{2.21}) and (\ref{2.25}), we have
\begin{eqnarray}
  \frac{2(1+\phi_{z}d)}{|Dw|^{2}}|w_{i}D_{i}f|&=&|\frac{2(1+\phi_{z}d)}{|Dw|^{2}}w_{i}(f_{x_i}+f_zu_i)-f_{p_i}(\alpha_0d_i+h'u_i+\frac{2w_lR_{ll}}{|Du|^2})|\nonumber\\
  &\leq& C(1+|Du|^{2k-1}),\label{4.41}
\end{eqnarray}
here we use the fact that $|R_{ll}|\leq C(\mu|Du|^2+|Du|+1)$.
Put (\ref{4.40}) and (\ref{4.41}) into (\ref{2.28}), we have
\begin{eqnarray}\label{2.29}
  0&\geq&\frac{2F^{ii}w_{li}^{2}}{|Dw|^{2}}-\frac{C\mu F^{ii}u_{ii}^{2}}{|Dw|^{2}}-C\mu |Dw|^{2}\mathcal{F}-C(|Dw|\mathcal{F}+1)\nonumber\\
  &&-C(1+|Du|^{2k-1}),
\end{eqnarray}
where $C=C(\alpha_{0}, M, n, m, |d|_{C^{3}}, L_{1}, L_{2}, \mu)$.

By (\ref{2.21}), (\ref{2.23}) and the inequality (see \cite{hmw})
\begin{eqnarray}
  (a+b)^{2}\geq \epsilon a^{2}-\frac{\epsilon}{1-\epsilon}b^{2},\nonumber
\end{eqnarray}
choose $\epsilon=\frac{1}{2}$, we obtain
\begin{eqnarray}
  w_{ii}^{2}&\geq&\frac{1}{4}u_{ii}^{2}-R_{ii}^{2}\nonumber\\
  &\geq& \frac{1}{4}u_{ii}^{2}-C(\mu^{2}|Dw|^{4}+|Dw|^{2}+1).\nonumber
\end{eqnarray}
It follows that
\begin{eqnarray}\label{2.32}
  0&\geq&(\frac{1}{8}-C\mu)\frac{F^{ii}u_{ii}^{2}}{|Dw|^{2}}-C\mu |Dw|^{2}\mathcal{F}-C(|Dw|+1)\mathcal{F}\nonumber\\
  &&-C(1+|Du|^{2k-1}).
\end{eqnarray}

There exists at least a index $l_0$ such that $u_{l_0}\geq \frac{|Du|}{\sqrt{n}}$. We rewrite the (\ref{2.25}) as
\begin{eqnarray}
  2w_{l_0}w_{l_0l_0}+2\sum_{q\neq l_0}w_{q}w_{ql_0}=-(\alpha_{0}d_{l_0}+h'u_{l_0})|Dw|^{2}.\nonumber
\end{eqnarray}
From (\ref{2.21}) we have
\begin{eqnarray}
  2(1+\phi_{z}d)w_{l_0}u_{l_0l_0}=-(\alpha_{0}d_{l_0}+h'u_{l_0})|Dw|^{2}-2w_{q}R_{ql}.
\end{eqnarray}
Since $|R_{l}|\leq 2L_{2}\leq \frac{u_{l}}{4}$, from (\ref{2.20}), we have $w_{l}\geq \frac{u_{l}}{4}$. If we assume that $|Du|\geq \frac{2\sqrt{n}\alpha_{0}|Dd|}{h'}$, and use the facts that $1+\phi_{z}d\geq\frac{1}{2}$ and $|R_{ij}|\leq C(\mu|Du|^{2}+|Du|+1)$, then
\begin{eqnarray}
 u_{l_0l_0}\leq -2h'|Dw|^{2}+12\sqrt{n}C(\mu|Dw|^{2}+|Dw|).\nonumber
\end{eqnarray}
If we assume that $|Dw|\geq \frac{2}{h'}\geq 10M+2$ and $\mu\leq\frac{h'}{12\sqrt{n}C}$, then
\begin{eqnarray}
  u_{l_0l_0}\leq -\frac{h'}{2}|Dw|^{2}.\label{h4.50}
\end{eqnarray}

Denote $u_{11}\geq\cdot\cdot\cdot\geq u_{nn}$. By (\ref{h4.50}), we can choose $\delta=\frac{h'}{2}$, $L=|Dw|^2$ and $\theta_1=\frac{C_n^m(h')^{k-1}}{4^{k-1}}$ in the Proposition \ref{pro7}, such that
\begin{eqnarray}
  u_{nn}\leq-\frac{h'}{2}|Dw|^{2},\quad F^{nn}\geq \frac{1}{n}\mathcal{F}\geq\frac{1}{n}\theta_1|Dw|^{2k-2}.\label{4.50}
\end{eqnarray}
We assume that $\mu\leq\min\{\frac{1}{16C},\frac{h'^{2}}{128nC}\}$. By  (\ref{2.32}) we obtain
\begin{eqnarray}
  0&\geq&\frac{h'^{2}}{128n}|Dw|^{2}\mathcal{F}-C(|Dw|+1)\mathcal{F}
 -C(1+|Du|^{2k-1}).\label{4.26}
\end{eqnarray}
By (\ref{4.50}), we have
\begin{eqnarray}
 0&\geq&\frac{h'^{2}}{128n}|Dw|^{2}-C|Dw|-C.
\end{eqnarray}
 It is easy to get a bound for  $|Dw|(x_0)$, then a bound for $G(x_0)$.

Anyway we have the bound
\begin{eqnarray}
  G(x_0)=\sup_{\overline{\Omega}_{\mu}}G(x)\leq C,\nonumber
\end{eqnarray}
where $C=C(\alpha_{0}, M, n, m, |d|_{C^{3}}, L_{1}, L_{2}, \mu)$. Thus we obtain
\begin{eqnarray}
  \sup_{\overline{\Omega}_{\mu}}|Du|\leq C+\log(1+2M)+\alpha_0\mu.
\end{eqnarray}
\end{proof}

By the same reason for Theorem \ref{th3}, we have the following boundary gradient estimate when $f=f(x,u)$.
\begin{theorem}\label{th4}
Suppose $\Omega\subset\mathbb{R}^n\ (n\geq3)$ is a bounded domain with $C^ 3$ boundary, and $2\leq k\leq C_n^m$,
Let $f(x,z)\in C^1(\overline{\Omega}\times [-M_0,M_0])$ is a nonnegative function and
$\phi\in C^3(\overline{\Omega}\times [-M_0,M_0])$, $M_0=\sup\limits_{\overline{\Omega}}|u|$.
We also assume that there exists constants $L_1$  and $L_2$ such that
\begin{eqnarray}
|f|_{C^1(\overline{\Omega}\times [-M_0,M_0])}&\leq& L_1,\\
|\phi|_{C^3(\overline{\Omega}\times [-M_0,M_0])}&\leq& L_2.
\end{eqnarray}
If $u \in C ^3 (\Omega)\cap C^1 (\overline{\Omega})$ is a $k$-admissible solution of  the equation
\begin{equation}
  \left\{
  \begin{aligned}
  S_{k}(W)&=f(x,u),\quad\text{in}\ \Omega,\\
  u_{\nu}&=\phi(x,u),\quad\text{on}\ \partial\Omega.
  \end{aligned}
  \right.
\end{equation}
 Then we have
\begin{eqnarray}
\sup_{\Omega_{\mu}}|Du|\leq C,
\end{eqnarray}
where $C$ is a constant depends only on $n$, $k$, $m$, $\mu$, $M_0$, $L_1$, $L_2$ and $\Omega$.
\end{theorem}
\begin{proof}
  By the same auxiliary function and the same computations as in the proof above, now we deal with terms in (\ref{4.41}) as follows
  \begin{eqnarray}
     \frac{2(1+\phi_{z}d)}{|Dw|^{2}}|w_{i}D_{i}f|&=&|\frac{2(1+\phi_{z}d)}{|Dw|^{2}}w_{i}(f_{x_i}+f_zu_i)|\nonumber\\
  &\leq& 4L_1(1+|Du|^{-1}).
  \end{eqnarray}
  It is not hard to get, a different version of (\ref{4.26}),
  \begin{eqnarray}
  0&\geq&\frac{h'^{2}}{128n}|Dw|^{2}\mathcal{F}-C(|Dw|+1)\mathcal{F}
 -C(1+|Du|^{-1}).
\end{eqnarray}
From (\ref{4.50}), we still have
\begin{eqnarray}
  u_{nn}\leq-\frac{h'}{2}|Dw|^{2},\quad F^{nn}\geq \frac{1}{n}\mathcal{F}.\nonumber
\end{eqnarray}
By the Newton-Maclaurin inequality, we have
\begin{eqnarray}
  F^{nn}\geq \frac{1}{n}\mathcal{F}\geq c S_k^{\frac{1}{k}}(\lambda)\geq c(\min{f})^{\frac{1}{k}},
\end{eqnarray}
where $c=c(n,m,k)$ a universal constant. Then we also have
\begin{eqnarray}
 0&\geq&\frac{h'^{2}}{128n}|Dw|^{2}-C|Dw|-C.
\end{eqnarray}
It is also give a bound for $|Dw|$ at interior maximum point of $G$.
Through the same discussion as before, we have
\begin{eqnarray}
  \sup_{\overline{\Omega}_{\mu}}|Du|\leq C+\log(1+2M)+\alpha_0\mu.
\end{eqnarray}
\end{proof}

\section{Global Second Order Derivatives Estimates}\label{sec5}
\medskip


 \subsection{Reduce the global second derivative estimates into double normal derivatives estimates on boundary}

Using the method of Lions-Trudinger-Urbas \cite{ltu}, we can reduce the second derivative estimates of the solution into the  boundary double normal estimates.

\begin{lemma}\label{th4.4}
 Let $\Omega\subset \mathbb{R}^n$ be a bounded domain with $C^4$ boundary.
 Assume $f(x,z)\in C^{2}(\overline{\Omega}\times \mathbb{R})$ is positive
 and $\phi(x,z)\in C^{3}(\overline{\Omega}\times\mathbb{R})$ with $\phi_z-2\kappa_{min}<0$.
 If $u\in C^{4}(\Omega)\cap C^{3}(\overline{\Omega})$ is a $k$-admissible
solution of the Neumann problem
\begin{equation}\label{eq5}
\left\{
\begin{aligned}
 S_{k}(W)={f}(x,u),&&\text{in}\ \ \Omega,\\
u_{\nu}=\phi(x,u),&&\text{on}\ \ \partial\Omega.
\end{aligned}
\right.
\end{equation}
 Denote $N=\sup\limits_{\partial\Omega}|u_{\nu\nu}|$, then
 \begin{eqnarray}\label{5.1}
 \sup_{\overline{\Omega}}|D^2u|\leq C_{0}(1+N).
 \end{eqnarray}
where $C_0$ depends on $n$, $m$, k, $|u|_{C^{1}(\overline{\Omega})}$,
$|f|_{C^{2}(\overline{\Omega}\times [-M_{0},M_{0}])}$,
$\min f$, $|\phi|_{C^{3}(\overline{\Omega}\times[-M_{0},M_{0}])}$ and $\Omega$. Here $M_{0}=\sup\limits_{\overline{\Omega}}|u|$.
\end{lemma}

\begin{proof}
Write equation (\ref{eq5}) in the form of
\begin{equation}\label{eq8}
\left\{
\begin{aligned}
 S_{k}^{\frac{1}{k}}(W)=\widetilde{f}(x,u),&&\text{in}\ \ \Omega,\\
u_{\nu}=\phi(x,u),&&\ \text{on}\ \ \partial\Omega.
\end{aligned}
\right.
\end{equation}
where $\widetilde{f}=f^{\frac{1}{k}}$.
Since $\lambda(W)\in\Gamma_k\subset\Gamma_2$ in $\mathbb{R}^{C_n^m}$, we have
 \begin{eqnarray}
   \sum_{i\neq j}|u_{ij}|\leq c(n,m)S_1(W)=mc(n,m)S_1(D^2u),\label{5.2}
 \end{eqnarray}
 where $c(n,m)$ is a universal number independent of $u$.
 Thus, it is sufficiently to prove (\ref{5.1}) for any direction $\xi\in\mathbb{S}^{n-1}$, that is
  \begin{eqnarray}
 u_{\xi\xi}\leq C_{0}(1+N).\label{5.3}
 \end{eqnarray}

 We consider the following auxiliary function in $\Omega\times \mathbb{S}^{n-1}$,
 \begin{eqnarray}
   v(x,\xi):=u_{\xi\xi}-v'(x,\xi)+K_1|x|^2+K_2|Du|^2,\label{5.22}
 \end{eqnarray}
 where $v'(x,\xi)=a^{l}u_{l}+b:= 2(\xi\cdot\nu)\xi'\cdot(\phi_{x_{l}}+\phi_{z}u_{l}-u_{l}D\nu^{l})$, with  $\xi'=\xi-(\xi\cdot\nu)\nu$ and $a^{l}=2(\xi\cdot\nu)(\xi'^{l}\phi_{z}-\xi'^{i}D_{i}\nu^{l})$. $K_{1}$, $K_{2}$ are positive constants to be determined. By a direct computation, we have
 By direct computations, we have
  \begin{eqnarray}
    v_{i}&=&u_{\xi\xi i}-D_{i}a^{l}u_{l}-a^{l}u_{ii}-D_{i}b+2K_{1}x_{i}+2K_{2}u_{l}u_{li},\label{5.25}\\
    v_{ij}&=& u_{\xi\xi ij}-D_{ij}a^{l}u_{l}-D_{i}a^{l}u_{lj}-D_{j}a^{l}u_{li}-a^{l}u_{lij}-D_{ij}b\nonumber\\
    &&+2K_{1}\delta_{ij}+2K_{2}u_{li}u_{lj}+2K_{2}u_{l}u_{lij}.\label{5.26}
    \end{eqnarray}

  Denote $\widetilde{F}(D^2u)=S_{k}^{\frac{1}{k}}(W)$, and
    \begin{eqnarray}
      \widetilde{F}^{ij}=\frac{\partial \widetilde{F}}{\partial u_{ij}}=\frac{\partial S_k^{\frac{1}{k}}(W)}{\partial    w_{\overline{\alpha}\overline{\beta}}}\frac{\partial w_{\overline{\alpha}\overline{\beta}}}{\partial u_{ij}},
    \end{eqnarray}
    and
    \begin{eqnarray}
      \widetilde{F}^{pq,rs}&=&\frac{\partial^{2} \widetilde{F}}{\partial u_{pq}\partial u_{rs}}\\
      &=&\frac{\partial^{2} S_k^{\frac{1}{k}}(W)}{\partial w_{\overline{\alpha}\overline{\beta}}\partial w_{\overline{\eta}\overline{\xi}}}\frac{\partial w_{\overline{\alpha}\overline{\beta}}}{\partial u_{pq}}\frac{\partial w_{\overline{\eta}\overline{\xi}}}{\partial u_{rs}},\nonumber
    \end{eqnarray}
    since $w_{\overline{\alpha}\overline{\beta}}$ is a linear combination of $u_{ij},\ 1\leq i,j\leq n$.
 Differentiating the equation (\ref{eq8}) twice, we have
 \begin{eqnarray}
   \widetilde{F}^{ij}u_{ijl}=D_l\widetilde{f},\label{5.4}
 \end{eqnarray}
 and
 \begin{eqnarray}
   \widetilde{F}^{pq,rs}u_{pq\xi}u_{rs\xi}+\widetilde{F}^{ij}u_{ij\xi\xi}
   =D_{\xi\xi}\widetilde{f}.\label{5.6}
 \end{eqnarray}
 By the concavity of $S_{k}^{\frac{1}{k}}(W)$ operator with respect to $W$, we have
    \begin{eqnarray}
      D_{\xi\xi}\widetilde{f}=
      \widetilde{F}^{pq,rs}u_{pq\xi}u_{rs\xi}+\widetilde{F}^{ij}u_{ij\xi\xi}
      \leq\widetilde{F}^{ij}u_{ij\xi\xi}.\label{5.12}
    \end{eqnarray}

 Now we contract (\ref{5.26}) with $\widetilde{F}^{ij}$ to get, using
 (\ref{5.4})-(\ref{5.12}),
 \begin{eqnarray}
      \widetilde{F}^{ij}v_{ij}&=&\widetilde{F}^{ij}u_{ij\xi\xi}
      -\widetilde{F}^{ij}D_{ij}a^{l}u_{l}-2\widetilde{F}^{ij}D_{i}a^{l}u_{lj}
      -\widetilde{F}^{ij}u_{ijl}a^{l}-\widetilde{F}^{ij}D_{ij}b\nonumber\\
      &&+2K_{1}\widetilde{\mathcal{F}}+2K_{2}\widetilde{F}^{ij}u_{il}u_{jl}
      +2K_{2}\widetilde{F}_{ij}u_{ijl}u_{l}\nonumber\\
      &\geq&D_{\xi\xi}\widetilde{f}-\widetilde{F}^{ij}D_{ij}a^{l}u_{l}
      -2\widetilde{F}^{ij}D_{i}a^{l}u_{ij}-a^{l}D_{l}\widetilde{f}
      -\widetilde{F}^{ij}D_{ij}b\nonumber\\
      &&+2K_{1}\widetilde{\mathcal{F}}+2K_{2}\widetilde{F}^{ij}u_{il}u_{jl}
      +2K_{2}u_{l}D_{l}\widetilde{f}.
    \end{eqnarray}
 where $\widetilde{\mathcal{F}}=\sum\limits_{i=1}^{n}\widetilde{F}^{ii}$.
 Note that
    \begin{eqnarray}
      D_{\xi\xi}\widetilde{f}=\widetilde{f}_{\xi\xi}+2\widetilde{f}_{\xi z}u_{\xi}+\widetilde{f}_{z}u_{\xi\xi},\nonumber\\
      D_{ij}a^{l}=2(\xi\cdot\nu)\xi'^{l}\phi_{zz}u_{ij}+r^{l}_{ij},\nonumber\\
      D_{ij}b=2(\xi\cdot\nu)\xi'^{l}\phi_{x_{l}z}u_{ij}+r_{ij},\nonumber
    \end{eqnarray}
    with $|r^{l}_{ij}|, |r_{ij}|\leq C(|u|_{C^{1}}, |\phi|_{C^{3}}, |\partial \Omega|_{C^{4}})$. At the maximum point $x_{0}\in\Omega$ of $v$, we can assume $(u_{ij})_{n\times n}$ is diagonal. It follows that, by the Cauchy-Schwartz inequality,
 \begin{eqnarray}\label{3.10}
      \widetilde{F}^{ij}v_{ij}&\geq&-C(\mathcal{\widetilde{F}}+K_2+1)
      -C\widetilde{F}^{ii}|u_{ii}|+\widetilde{f}_{z}u_{\xi\xi}
      +2K_{1}\widetilde{\mathcal{F}}+2K_{2}\widetilde{F}^{ii}u_{ii}^2\nonumber\\
      &\geq&-C(\mathcal{\widetilde{F}}+K_2+1)+\widetilde{f}_{z}u_{\xi\xi}
      +2K_{1}\widetilde{\mathcal{F}}+(2K_{2}-1)\widetilde{F}^{ii}u_{ii}^2,
    \end{eqnarray}
    where $C=C(|u|_{C^{1}}, |\phi|_{C^{3}}, |\partial \Omega|_{C^{4}}, |f|_{C^{2}})$.

 Assume $u_{11}\geq u_{22}\cdots\geq u_{nn}$, and denote $\lambda_{1}\geq\lambda_{2}\geq\cdots\geq\lambda_{C_{n}^{m}}$ the eigenvalues of the matrix $(w_{\overline{\alpha}\overline{\beta}})_{C_n^m\times C_n^m}$.
It is easy to see    $\lambda_{1}=u_{11}+\sum\limits_{i=2}^mu_{ii}\leq mu_{11}$. Then we have, by (\ref{sw1}) in Proposition \ref{pro2} and (\ref{2.8}) in Proposition \ref{pro3},
  \begin{eqnarray}
    \widetilde{F}^{11}u_{11}^{2}&=&\sum_{1\in\overline{\alpha}}
    \frac{1}{k}S_k^{\frac{1}{k}-1} S_{k-1}(\lambda|N_{\overline{\alpha}})u_{11}^{2}\nonumber\\
    &\geq&\frac{1}{mk}S_k^{\frac{1}{k}-1} S_{k-1}(\lambda|1)\lambda_{1}u_{11}\nonumber\\
    &\geq&\frac{1}{m C_n^m}S_{k}^{\frac{1}{k}}u_{11}=\frac{\widetilde{f}}{m C_n^m}u_{11}.\label{5.21}
  \end{eqnarray}
 We can assume $u_{\xi\xi}\geq 0$, otherwise we have (\ref{5.3}). Plug (\ref{5.21}) into (\ref{3.10}) and use the Cauchy-Schwartz inequality, then
\begin{eqnarray}
  \widetilde{F}^{ii}v_{ii}&\geq&(K_{2}-1)\sum_{i=1}^{n}\widetilde{F}^{ii}u_{ii}^{2}
  +(\frac{K_{2}\widetilde{f}}{mC_n^m}+\widetilde{f}_{z})u_{\xi\xi}
  +(2K_{1}-C)\mathcal{\widetilde{F}}
  \nonumber\\&&-C(K_2+1).
\end{eqnarray}
Choose $K_{2}=\frac{mC_n^m|\max f_{z}|}{k\min f}+1$ and $K_{1}=C(K_2+2)+1$. It follows that
\begin{eqnarray}
    \widetilde{F}^{ii}v_{ii}&\geq&(2K_{1}-C)\widetilde{\mathcal{F}}-C(K_2+1)>0,
\end{eqnarray}
since we have $\widetilde{\mathcal{F}}\geq 1$ from (\ref{2.10}). This implies that $v(x,\xi)$ attains its maximum on the boundary by the maximum principle. Now we assume $(x_{0},\xi_{0})\in \partial\Omega\times \mathbb{S}^{n-1}$ is the maximum pint of $v(x,\xi)$ in $\overline{\Omega}\times\mathbb{S}^{n-1}$.  Then we consider two cases as follows,

$\mathbf{Case1}$. $\xi_{0}$ is a tangential vector at $x_{0}\in\partial\Omega$.

We directly have $\xi_0\cdot\nu=0$ , $\nu=-Dd$, $v'(x_0,\xi_0)=0$, and $u_{\xi_0, \xi_0}(x_0)>0$. As in \cite{l}, we define
\begin{eqnarray}
  c^{ij}=\delta_{ij}-\nu^i\nu^j,\quad\ \text{in}\ \Omega_{\mu},
\end{eqnarray}
and it is easy to see that $c^{ij}D_j$ is a tangential direction  on $\partial\Omega$.
We compute at $(x_0, \xi_0)$.

From the boundary condition, we have
\begin{eqnarray}
  u_{li}\nu^l&=&(c^{ij}+\nu^i\nu^j)\nu^lu_{lj}\nonumber\\
  &=&c^{ij}u_j\phi_z+c^{ij}\phi_{x_j}-c^{ij}u_lD_j\nu^l+\nu^i\nu^j\nu^lu_{lj}.\label{5.16}
\end{eqnarray}
It follows that
\begin{eqnarray}
  u_{lip}\nu^l&=&[c^{pq}+\nu^p\nu^q]u_{liq}\nu^l\nonumber\\
  &=&c^{pq}D_q(c^{ij}u_j\phi_z+c^{ij}\phi_{x_j}-c^{ij}u_lD_j\nu^l+\nu^i\nu^j\nu^lu_{lj})-c^{pq}u_{li}D_q\nu^l+\nu^p\nu^q\nu^lu_{liq},\nonumber
\end{eqnarray}
then we obtain
\begin{eqnarray}
  u_{\xi_0\xi_0\nu}&=&\sum_{ilp=1}^{n}\xi_0^i\xi_0^pu_{lip}\nu^l\nonumber\\
  &=&\sum_{i=1}^{n}\xi_0^i\xi_0^q[D_q(c^{ij}u_j\phi_z+c^{ij}\phi_{x_j}-c^{ij}u_lD_j\nu^l+\nu^i\nu^j\nu^lu_{lj})-u_{li}D_q\nu^l]\nonumber\\
    &\leq& \phi_z u_{\xi_0\xi_0}-2\xi_0^i\xi_0^q u_{li}D_q\nu^l+C(1+|u_{\nu\nu}|).
\end{eqnarray}
We assume $\xi_0=e_1$, it is easy to get the bound for $u_{1i}(x_0)$ for $i>1$ from the maximum of $v(x,\xi)$ in the $\xi_0$ direction. In fact, we can assume $\xi(t)=\frac{(1, t, 0,\cdots, 0)}{\sqrt{1+t^2}}$. Then we have
\begin{eqnarray}
  0&=&\frac{dv(x_0,\xi(t))}{dt}|_{t=0}\nonumber\\
  &=&2u_{12}(x_0)-2\nu^2(\phi_zu_1-u_lD_l\nu^l),\nonumber
\end{eqnarray}
so
\begin{eqnarray}
  |u_{12}|(x_0)\leq C+C|Du|.\label{5.41}
\end{eqnarray}
Similarly, we have for $\forall i>1$,
\begin{eqnarray}
  |u_{1i}|(x_0)\leq C+C|Du|.\label{5.42}
\end{eqnarray}
Thus we have, by $D_1\nu^1\geq\kappa_{min}$,
\begin{eqnarray}
  u_{\xi_0\xi_0\nu}&\leq&\phi_z u_{\xi_0\xi_0}-2D_1\nu^1u_{11}+C(1+|u_{\nu\nu}|)\nonumber\\
  &\leq&(\phi_z-2\kappa_{min}) u_{\xi_0\xi_0}+C(1+|u_{\nu\nu}|).\nonumber
\end{eqnarray}
On the other hand, we have from the Hopf lemma, (\ref{5.25}) and (\ref{5.42}),
\begin{eqnarray}
  0&\leq&v_{\nu}(x_0,\xi_0)\nonumber\\
  &=&u_{\xi_0\xi_0 \nu}-D_{\nu}a^{l}u_{l}-a^{l}u_{\nu\nu}-D_{\nu}b+2K_{1}x_{i}\nu^i+2K_{2}u_{l}u_{l\nu}\nonumber\\
    &\leq&(\phi_z-2\kappa_{min})u_{\xi_0\xi_0}+C(1+|u_{\nu\nu}|).\nonumber
\end{eqnarray}
Then we get, since $2\kappa_{min}-\phi_z\geq c>0$,
\begin{eqnarray}
  u_{\xi_0\xi_0}(x_0)\leq C(1+|u_{\nu\nu}|).
\end{eqnarray}

   \textbf{Case2.} $\xi_0$ is non-tangential.

   We can find a tangential vector $\tau$, such that $\xi_0 = \alpha\tau+\beta\nu$,
with $\alpha^2 + \beta^2 = 1$. Then we have
\begin{eqnarray}
  u_{\xi_0\xi_0}(x_0)&=&\alpha^2u_{\tau\tau}(x_0)+\beta^2u_{\nu\nu}(x_0)+2\alpha\beta u_{\tau\nu}(x_0)\nonumber\\
  &=&\alpha^2u_{\tau\tau}(x_0)+\beta^2u_{\nu\nu}(x_0)+2(\xi_0\cdot\nu)\xi'_0\cdot(\phi_zDu-u_lD\nu^l).\nonumber
\end{eqnarray}
By the definition of $v(x_0,\xi_0)$,
\begin{eqnarray}
  v(x_0,\xi_0)&=&\alpha^2v(x_0,\tau)+\beta^2v(x_0,\nu)\nonumber\\
  &\leq&\alpha^2v(x_0,\xi_0)+\beta^2v(x_0,\nu).\nonumber
\end{eqnarray}
Thus,
\begin{eqnarray}
  v(x_0,\xi_0)=v(x_0,\nu),\nonumber
\end{eqnarray}
and
\begin{eqnarray}
  u_{\xi_0\xi_0}(x_0)\leq |u_{\nu\nu}|+C.
\end{eqnarray}
In conclusion, we have (\ref{5.3}) in both cases.
 \end{proof}


\subsection{Global second order estimates by double normal  estimates on boundary}

\medskip
Generally, the double normal estimates are the most important and hardest  parts for the Neumann problem.
As in \cite{ltu} and \cite{mq},
we construct sub and super barrier function to give lower and upper bounds for $u_{\nu\nu}$ on the boundary. Then we give the global
second order estimates.

\subsubsection{\bf Global second order estimate for Theorem \ref{th1.1}}
In this subsection, we establish the following global second order estimate.
\begin{theorem}\label{th5.1}
Let $\Omega\subset\mathbb{R}^{n}$ be a bounded domain with $C^{4}$ boundary,  $2\leq m\leq n-1$, and $2\leq k\leq C_{n-1}^{m-1}$.
Assume $f(x,z)\in C^{2}(\overline{\Omega}\times \mathbb{R})$ is positive
and $\phi(x,z)\in C^{3}(\overline{\Omega}\times\mathbb{R})$ with $\phi_z-2\kappa_{min}<0$.
If $u\in C^{4}(\Omega)\cap C^{3}(\overline{\Omega})$ is a $k$-admissible
solution of the Neumann problem (\ref{eq5}).
 Then we have
\begin{eqnarray}
\sup_{\overline{\Omega}}|D^{2}u|\leq C,\label{5.5}
\end{eqnarray}
where $C$ depends only on $n$, $m$, k, $|u|_{C^{1}(\overline{\Omega})}$,$|f|_{C^{2}(\overline{\Omega}\times[-M_{0},M_{0}])}$, $\min f$, $|\phi|_{C^{3}(\overline\Omega\times[-M_{0},M_{0}])}$ and $\Omega$, where $M_{0}=\sup\limits_{\Omega}|u|$.
\end{theorem}

First, we denote $d(x)=dist(x,\partial\Omega)$, and define
\begin{eqnarray}
  h(x):=-d(x)+K_3d^{2}(x).\label{h1}
\end{eqnarray}
where $K_3$ is large constant to be determined later. Then we give the following key Lemma.

\begin{lemma}\label{le5.1}
  Suppose $\Omega\subset\mathbb{R}^n$ is a bounded domain  with $C^2$ boundary, $2\leq m\leq n-1$ and $2\leq k\leq C_{n-1}^{m-1}$.  Let  $u\in C^2(\overline{\Omega})$ is a k-admissible solution of the equation (\ref{eq})and $h$ is defined as in (\ref{h1}).
 Then, there exists $K^{*}$, a sufficiently large number depends only on $n$, $m$, $k$, $\min f$ and  $\Omega$, such that,
  \begin{eqnarray}
    F^{ij}h_{ij}\geq K_3^{\frac{1}{2}}(1+\mathcal{F}),\quad\text{in}\ \Omega_{\mu}\ (0<\mu\leq\widetilde{\mu}),\label{h3}
  \end{eqnarray}
  for any $K_3\geq K^{*}$, where $\widetilde{\mu}=\min\{\frac{1}{4K_3},\mu_0\}$, $\mu_0$ is mentioned in (\ref{d}).
\end{lemma}

\begin{proof}
   For $x_0\in\Omega_{\mu}$, there exists $y_0\in\partial\Omega$ such that $|x_0-y_0|=d(x_0)$. Then, in terms of a principal coordinate system at $y_0$, we have (see \cite{gt}, Lemma 14.17),
  \begin{eqnarray}
    [D^2d(x_0)]=-diag\big[\frac{\kappa_1}{1-\kappa_1d},\cdots,\frac{\kappa_{n-1}}{1-\kappa_{n-1}d},0\big],\label{d2}
  \end{eqnarray}
  and
  \begin{eqnarray}
    Dd(x_0)=-\nu(x_0)=(0,\cdots,0,-1).
  \end{eqnarray}
  Observe that
  \begin{eqnarray}
    [D^2h(x_0)]=diag\big[\frac{((1-2K_3d)\kappa_1}{1-\kappa_1d},\cdots,\frac{(1-2K_3d)\kappa_{n-1}}{1-\kappa_{n-1}d},2K_3\big].\label{h2}
  \end{eqnarray}
  Denote $\mu_i=\frac{(1-K_3d)\kappa_i}{1-\kappa_id}, \ \forall 1\leq i\leq n-1$, and $\mu_n=2K_3$ for simplicity. Then we define $\lambda(D^2h)=\{\mu_{i_1}+\cdots+\mu_{i_m}|\ 1\leq i_1<\cdots<i_m\leq n\}$ and assume $\lambda_1\geq\cdots\geq\lambda_{C_n^m}$, it is easy to see
  \begin{eqnarray}
    \lambda_k\geq\lambda_{C_{n-1}^{m-1}}\geq 2K_3+\sum_{l=1}^{m-1}\mu_{i_l}\geq K_3,\nonumber
  \end{eqnarray}
  if we choose $K_3$ sufficiently large and $\mu\leq\frac{1}{4K_3}$. It follows that, for $\forall 1\leq l\leq k$,
  \begin{eqnarray}
    S_l(\lambda)&\geq& K_3^l-C(n,m,\kappa)K_3^{l-1}\nonumber\\
    &\geq& \frac{K_3^l}{2},
  \end{eqnarray}
  such that $h$ is $k$-admissible. Similarly, $w=h-\frac{K_3}{2n}|x|^2$ is also $k$-admissible if we choose $K_3$ sufficiently large. By the concavity of $\widetilde{F}$, we have
  \begin{eqnarray}
    \widetilde{F}^{ij}w_{ij}&\geq&\widetilde{F}[D^2u+D^2w]-\widetilde{F}[D^2u]\nonumber\\
    &\geq&\widetilde{F}[D^2w]\nonumber\\
    &\geq&\frac{K_3}{4}.
  \end{eqnarray}
  Then we have
  \begin{eqnarray}
    \widetilde{F}^{ij}h_{ij}=\widetilde{F}^{ij}(h-\frac{K_3}{2n}|x|^2+\frac{K_3}{2n}|x|^2)_{ij}\geq\frac{K_3}{4n}(1+\widetilde{\mathcal{F}}).
  \end{eqnarray}
  If we choose $K_3\geq (\frac{4n\max f^{\frac{1}{k}}}{k\min f})^2$, then we have
  \begin{eqnarray}
    F^{ij}h_{ij}\geq K_3^{\frac{1}{2}}(1+\mathcal{F}).
  \end{eqnarray}
\end{proof}

Now we can use Lemma \ref{le5.1} to prove Theorem \ref{th5.1}

\begin{proof}[\bf Proof of Theorem \ref{th5.1}]
 We define
 \begin{eqnarray}
   P(x)=Du\cdot\nu-\phi(x,u),
 \end{eqnarray}
 with $\nu=-Dd$.
 Differentiate $P$ twice to  obtain

  \begin{eqnarray}
    P_{ij}&=&-u_{rij}d_{r}-u_{ri}d_{rj}-u_{rj}d_{ri}-u_{r}d_{rij}-D_{ij}\phi.\label{p2}
  \end{eqnarray}
 Then we obtain
  \begin{eqnarray}\label{4.23}
    F^{ij}P_{ij}&=& -F^{ij}(u_{rij}d_{r}+2u_{ri}d_{rj}+u_{r}d_{rij}-D_{ij}\phi) \nonumber\\
    &\leq& -F^{ii}u_{ii}d_{ii}+C_1(1+\mathcal{F}),\nonumber
  \end{eqnarray}
where $C_1=C_{1}(|u|_{C^{1}}, |\partial \Omega|_{C^{3}}, |\phi|_{C^{2}}, |f|_{C^{1}}, n)$.
From (\ref{5.1}) in Lemma \ref{th4.4}, we have
\begin{eqnarray}
  |u_{ii}|\leq C_0(1+N).\nonumber
\end{eqnarray}
It follows that
\begin{eqnarray}
  F^{ij}P_{ij}&\leq& C_2(1+N)(1+\mathcal{F}),
\end{eqnarray}
where $C_2=C_1+C_0|d|_{C^2}$.

On the other hand, using Lemma \ref{le5.1}, we have
\begin{eqnarray}
  (A+\frac{1}{2}N)F^{ij}h_{ij}&\geq&(A+\frac{1}{2}N)K_3^{\frac{1}{2}}(1+\mathcal{F})\nonumber\\
  &\geq&
  C_2(1+N)(1+\mathcal{F})\nonumber\\
  &\geq&F^{ij}P_{ij},
\end{eqnarray}
if we choose $K_3=K^{*}+(2C_2)^2+1$ and $A\geq C_2+1$.

On $\partial \Omega$, it is easy to see
\begin{eqnarray}
  P=0.
\end{eqnarray}
On $\partial\Omega_{\mu}\cap\Omega$, we have
\begin{eqnarray}
  |P|\leq C_3(|u|_{C^{1}}, |\phi|_{C^{0}})\leq (A+\frac{1}{2}N)\frac{\mu}{2},
\end{eqnarray}
if we take $A=\max \{\frac{2C_3}{\mu},C_2+1\}$.

Finally the maximum principle tells us that
\begin{eqnarray}
  -(A+\frac{1}{2}N)h(x)\leq P(x)\leq (A+\frac{1}{2}N)h(x),\quad\text{in}\quad\Omega_{\mu}.
\end{eqnarray}

Suppose $u_{\nu\nu}(y_{0})=\sup\limits_{\partial\Omega}u_{\nu\nu}>0$, we have
\begin{eqnarray}
  0&\geq&P_{\nu}(y_{0})-(A+\frac{1}{2} N)h_{\nu}\nonumber\\
  &=&u_{\nu\nu}-D_{\nu}\phi-(A+\frac{1}{2} N)\nonumber\\
  &\geq&u_{\nu\nu}(y_{0})-C(|u|_{C^{1}}, |\partial\Omega|_{C^{2}}, |\phi|_{C^{2}})-(A+\frac{1}{2} N).\nonumber
\end{eqnarray}
Then we get
\begin{eqnarray}
  \sup_{\partial\Omega}u_{\nu\nu}\leq C+\frac{1}{2} N.\label{5.7}
\end{eqnarray}
Similarly, doing this at the minimum point of $u_{\nu\nu}$, we have
\begin{eqnarray}
  \inf_{\partial\Omega}u_{\nu\nu}\leq C+\frac{1}{2} N.\label{5.8}
\end{eqnarray}
It follows that
\begin{eqnarray}
  \sup_{\partial\Omega}|u_{\nu\nu}|\leq C.\label{5.9}
\end{eqnarray}
Combining (\ref{5.9}) with (\ref{5.1}) in Lemma \ref{th4.4}, we obtain
\begin{eqnarray}
  \sup_{\overline{\Omega}}|D^2u|\leq C.
\end{eqnarray}
\end{proof}

\subsubsection{\bf Global second order estimate for Theorem \ref{th0}}

In this subsection we give a global second order estimate for the cases that $m\leq\frac{n}{2}$. We can settle more cases for $k\geq C_{n-1}^{m-1}$
than before, if $\Omega$ is strictly $(m,k_0)$-convex.
\begin{theorem}\label{th5.2}
Let $\Omega\subset\mathbb{R}^{n}$ be a strictly $(m,k_0)$-convex domain with $C^{4}$ boundary,
 $2\leq m\leq \frac{n}{2}$, and $k=C_{n-1}^{m-1}+k_0\leq \frac{n-m}{n}C_{n}^{m}$.
 Assume $f(x,z)\in C^{2}(\overline{\Omega}\times \mathbb{R})$ is positive
 and $\phi(x,z)\in C^{3}(\overline{\Omega}\times\mathbb{R})$ with $\phi_z-2\kappa_{min}<0$.
 If $u\in C^{4}(\Omega)\cap C^{3}(\overline{\Omega})$ is a $k$-admissible
  solution of the Neumann problem (\ref{eq5}).
 Then we have
\begin{eqnarray}
\sup_{\overline{\Omega}}|D^{2}u|\leq C,
\end{eqnarray}
where $C$ depends only on $n$, $m$, k, $|u|_{C^{1}(\overline{\Omega})}$,
$|f|_{C^{2}(\overline{\Omega}\times[-M_{0},M_{0}])}$, $\min f$, $|\phi|_{C^{3}(\overline\Omega\times[-M_{0},M_{0}])}$
and $\Omega$, where $M_{0}=\sup\limits_{\overline{\Omega}}|u|$.
\end{theorem}

First, we prove the following Lemma.
  \begin{lemma}\label{le5.2}
  Let $\Omega\subset\mathbb{R}^{n}$ be a strictly $(m,k_0)$-convex domain with $C^{4}$ boundary,  $2\leq m\leq n-1$,
  and $k=C_{n-1}^{m-1}+k_0\leq \frac{n-m}{n}C_{n}^{m}$, $k_0$ a positive integer.  Assume  $u\in C^2(\overline{\Omega})$ is a k-admissible solution of the equation (\ref{eq})
  and $h$ is defined as in (\ref{h1}).
 Then, there exists $K_3$, a sufficiently large number depends only on $n$, $m$, $k$, $\min f$ and  $\Omega$, such that,
  \begin{eqnarray}
    F^{ij}h_{ij}\geq k_3(1+\mathcal{F}),\quad\text{in}\ \Omega_{\mu}\ (0<\mu\leq\widetilde{\mu}),\label{h4}
  \end{eqnarray}
  for $k_3$, a sufficiently
  small number depends only on $n$, $m$, $k$, and  $\Omega$. Here $\widetilde{\mu}=\min\{\frac{1}{4K_3},\mu_0\}$.
\end{lemma}
\begin{proof}

 For $x_0\in\Omega_{\mu}$, there exists $y_0\in\partial\Omega$ such that $|x_0-y_0|=d(x_0)$.
  As before, in terms of a principal coordinate system at $y_0$, we have,
  \begin{eqnarray}
    [D^2h(x_0)]=diag\big[\frac{((1-2K_3d)\kappa_1}{1-\kappa_1d},\cdots,\frac{(1-2K_3d)\kappa_{n-1}}{1-\kappa_{n-1}d},2K_3\big].
  \end{eqnarray}
  Denote $\mu_i=\frac{(1-K_3d)\kappa_i}{1-\kappa_id}, \ \forall 1\leq i\leq n-1$, and $\mu_n=2K_3$ for simplicity.
  Then we define $\lambda(D^2h)=\{\mu_{i_1}+\cdots+\mu_{i_m}|\ 1\leq i_1<\cdots<i_m\leq n\}$ and assume $\lambda_1\geq\cdots\geq\lambda_{C_n^m}$, it is easy to see
  \begin{eqnarray}
  \lambda_{C_{n-1}^{m-1}}\geq 2K_3+\sum_{l=1}^{m-1}\mu_{i_l}\geq \frac{3}{2}K_3,\nonumber
  \end{eqnarray}
  if we choose $K_3$ sufficiently large and $\mu\leq\frac{1}{4K_3}$.
  Then we denote $\lambda'=(\lambda_1,\cdots,\lambda_{C_{n-1}^{m-1}})$ and $\lambda(\kappa)=(\lambda_{C_{n-1}^{m-1}+1},\cdots,\lambda_{C_n^m})$.
  Since $\kappa\in\Gamma_{k_0}^{(m)}$, we have $\lambda(\kappa)\in\Gamma_{k_0}$ and $S_{k_0}(\lambda(\kappa))\geq b_0>0$.
  Then for $\forall\ 1\leq l\leq C_{n-1}^{m-1}$, we have
  \begin{eqnarray}
    S_l(\lambda)&\geq& S_l(\lambda')-c(n,m,k,\kappa)K_3^{l-1}\nonumber\\
    &\geq& K_3^l>0,
  \end{eqnarray}
  and, for $\forall\ l=C_{n-1}^{m-1}+l_0\leq k$, $l_0\leq k_0$,
  \begin{eqnarray}
    S_l(\lambda)&\geq& (\frac{3K_3}{2})^{C_{n-1}^{m-1}}S_{l_0}(\lambda(\kappa))-c(n,m,k,\kappa)K_3^{C_{n-1}^{m-1}-1}\nonumber\\
    &\geq& b_0^{\frac{l_0}{k_0}}(\frac{3K_3}{4})^{C_{n-1}^{m-1}}>0,
  \end{eqnarray}
  if we choose $K_3$ sufficiently large.
  It implies that $h$ is $k$-admissible. Similarly, $w=h-k_3|x|^2$
  is also $k$-admissible if $k_3$ sufficiently small.
  By the concavity of $\widetilde{F}$, we have
  \begin{eqnarray}
    \widetilde{F}^{ij}w_{ij}&\geq&\widetilde{F}[D^2u+D^2w]-\widetilde{F}[D^2u]\nonumber\\
    &\geq&\widetilde{F}[D^2w]\nonumber\\
    &\geq& \frac{1}{2}b_0^{\frac{1}{k}}(\frac{3K_3}{4})^{\gamma},
  \end{eqnarray}
  where $\gamma=\frac{C_{n-1}^{m-1}}{k}\leq 1$.

  Then we have
  \begin{eqnarray}
    \widetilde{F}^{ij}h_{ij}=\widetilde{F}^{ij}(h-k_3|x|^2+k_3|x|^2)_{ij}
    \geq \frac{1}{2}b_0^{\frac{1}{k}}(\frac{3K_3}{4})^{\gamma}+k_3\widetilde{\mathcal{F}},
  \end{eqnarray}
  for a large $K_3$.
  If we choose $K_3\geq 2(\frac{k_3\max f^{\frac{1}{k}}}{kb_0{\frac{1}{k}}\min f})^{\frac{1}{\gamma}}$, then we have
  \begin{eqnarray}
    F^{ij}h_{ij}\geq k_3(1+\mathcal{F}).
  \end{eqnarray}
\end{proof}

Following the line of Qiu and Ma \cite{mq}, we construct the sub barrier function as
\begin{eqnarray}
  P(x):=g(x)(Du\cdot Dh(x)-\psi(x))-G(x).
\end{eqnarray}
with
\begin{eqnarray}
  g(x)&:=&1-\beta h(x),\nonumber\\
  G(x)&:=&(A+\sigma N)h(x),\nonumber\\
  \psi(x)&:=& \phi(x, u)|Dh|(x),\nonumber
\end{eqnarray}
where $K_3$ is the constant in the following Lemma \ref{le4.3}, and $A$, $\sigma$,
 $\beta$ are positive constants to be determined.
We have the following lemma.
\begin{lemma}\label{le4.3}
Fix $\sigma$, if we select $\beta$ large, $\mu$ small, and $A$ large, then
  \begin{eqnarray}
    P\geq0,\quad\text{in}\quad\Omega_{\mu}.
  \end{eqnarray}
  Furthermore, we have
  \begin{eqnarray}\label{4.21}
    \sup_{\partial \Omega}u_{\nu\nu}\leq C+\sigma N,
  \end{eqnarray}
  where constant $C$ depends  only on $|u|_{C^{1}}$, $|\partial\Omega|_{C^{2}}$ $|f|_{C^2}$ and $|\phi|_{C^{2}}$.
\end{lemma}
\begin{proof}
  We assume $P(x)$ attains its minimum point $x_{0}$ in the interior of $\Omega_{\mu}$. Differentiate $P$ twice to  obtain
  \begin{eqnarray}
    P_{i}=g_{i}(u_{r}h_{r}-\psi)+g(u_{ri}h_{r}+u_{r}h_{ri}-\psi_{i})-G_{i},
  \end{eqnarray}
  and
  \begin{eqnarray}
    P_{ij}&=&g_{ij}(u_{r}h_{r}-\psi)+g_{i}(u_{rj}h_{r}+u_{r}h_{rj}-\psi_{j})\\
    &&+g_{j}(u_{ri}h_{r}+u_{r}h_{ri}-\psi_{i})+g(u_{rij}h_{r}+u_{ri}h_{rj}\nonumber\\
    &&+u_{rj}h_{ri}+u_{r}h_{rij}-\psi_{ij})-G_{ij}.\nonumber
  \end{eqnarray}

  By a rotation of coordinates, we may assume that $(u_{ij})_{n\times n}$ is diagonal at $x_{0}$, so are $W$ and $(F^{ij})_{n\times n}$. Denote $\mathcal{F}=\sum\limits_{i=1}^{n}F^{ii}$ the trace of $(F^{ij})_{n\times n}$. We choose $\mu<\min\{\frac{1}{4K_3},\frac{1}{\beta}\}$ so that $|\beta h|\leq \beta\frac{\mu}{2}\leq\frac{1}{2}$. It follows that
  \begin{eqnarray}\label{4.22}
    1\leq g\leq\frac{3}{2}.
  \end{eqnarray}
  By a straight computation we obtain
  \begin{eqnarray}\label{5.23}
    F^{ij}P_{ij}&=&F^{ii}g_{ii}(u_{r}h_{r}-\psi)+2F^{ii}g_{i}(u_{ii}h_{i}+u_{r}h_{ri}-\psi_{i})\nonumber\\
     &&+g F^{ii}(u_{rii}h_{r}+2u_{ii}h_{ii}+u_{r}h_{rii}-\psi_{ii})-(A+\sigma N)F^{ii}h_{ii} \nonumber\\
    &\leq&\big( \beta C_1-(A+\sigma N)k_{3}\big)(\mathcal{F}+1)\\
   && -2\beta F^{ii}u_{ii}h_{i}^{2}+2gF^{ii}u_{ii}h_{ii},\nonumber
  \end{eqnarray}
where $C_1=C_{1}(|u|_{C^{1}}, |\partial \Omega|_{C^{3}}, |\phi|_{C^{2}}, |f|_{C^{1}}, n)$.

  We divide indexes $I=\{1, 2, \cdots, n\}$ into two sets in the following way,
  \begin{eqnarray}
    B=\{i\in I | |\beta h_{i}^{2}|<\frac{k_{1}}{4}\},\nonumber\\
    G =I\backslash B =\{i\in I | |\beta h_{i}^{2}|\geq\frac{k_{1}}{4}\},\nonumber
  \end{eqnarray}
  where $k_1$ is a positive number depends on $|\partial\Omega|_{C^2}$ and $K_3$ such that $|D^2h|_{C^0}\leq\frac{k_1}{2}$.
  For $i\in G$, by $P_{i}(x_{0})=0$, we get
  \begin{eqnarray}
    u_{ii}=\frac{A+\sigma N}{g}+\frac{\beta(u_{r}h_{r}-\psi)}{g}-\frac{u_{r}h_{ri}-\psi_{i}}{h_{i}}.
  \end{eqnarray}
  Because $|h_{i}^{2}|\geq\frac{k_{1}}{4\beta}$ and (\ref{4.22}), we have
  \begin{eqnarray}
    |\frac{\beta(u_{r}h_{r}-\psi)}{g}-\frac{u_{r}h_{ri}-\psi_{i}}{h_{i}}|\leq\beta C_{2}(k_2, |u|_{C^{1}},
    |\partial \Omega|_{C^{2}}, |\psi|_{C^{1}}).\nonumber
  \end{eqnarray}
  Then let $A\geq3\beta C_{2}$, we have
  \begin{eqnarray}
    \frac{A}{3}+\frac{2\sigma N}{3}\leq u_{ii}\leq \frac{4A}{3}+\sigma N,
  \end{eqnarray}
  for $\forall i\in G$. We choose $\beta\geq4nk_{1}+1$ to let $|h_{i}^{2}|\leq\frac{1}{4n}$ for $i\in B$. Because $\frac{1}{2}\leq|Dh|\leq 2$, there is a $i_{0}\in G$, say $i_{0}=1$, such that
  \begin{eqnarray}
    h_{1}^{2}\geq \frac{1}{4n}.
  \end{eqnarray}

  We have
  \begin{eqnarray}\label{4.30}
    -2\beta \sum_{i\in I}F^{ii}u_{ii}h_{i}^{2}&=&-2\beta\sum_{i\in G}F^{ii}u_{ii}h_{i}^{2}-2\beta \sum_{i\in B}F^{ii}u_{ii}h_{i}^{2}\\
    &\leq& -2\beta F^{11}u_{11}h_{1}^{2}-2\beta\sum_{u_{ii}<0}F^{ii}u_{ii}h_{i}^{2}\nonumber\\
    &\leq&-\frac{\beta F^{11}u_{11}}{2n}-\frac{k_{1}}{2}\sum_{u_{ii}<0}F^{ii}u_{ii}.\nonumber
  \end{eqnarray}
  and
  \begin{eqnarray}\label{4.31}
    2g\sum_{i\in I}F^{ii}u_{ii}h_{ii}&=&2g\sum_{u_{ii}\geq 0}F^{ii}u_{ii}h_{ii}+2g\sum_{u_{ii}<0}F^{ii}u_{ii}h_{ii}\\
    &\leq&k_{1}\sum_{u_{ii}\geq0}F^{ii}u_{ii}-\frac{k_{1}}{2}\sum_{u_{ii}<0}F^{ii}u_{ii}.\nonumber
  \end{eqnarray}
  Plug (\ref{4.30}) and (\ref{4.31}) into (\ref{5.23}) to get
  \begin{eqnarray}\label{4.32}
    F^{ii}P_{ij}&\leq&\big(\beta C_{1}-(A+\sigma N)k_{3}\big)(\mathcal{F}+1)-\frac{\beta}{2n}F^{11}u_{11}\nonumber\\
    &&-k_{1}\sum_{u_{ii}<0}F^{ii}u_{ii}
    +k_{1}\sum_{u_{ii}\geq0}F^{ii}u_{ii}.
  \end{eqnarray}

  Denote $u_{22}\geq\cdots\geq u_{nn}$, and $\mu_i=u_{ii}\ (1\leq i\leq n)$ for simplicity. We also denote
  \begin{eqnarray}
    \lambda_{1}&=&\max\limits_{1\in\overline{\alpha}}\{w_{\overline{\alpha}\overline{\alpha}}\}
    =\mu_1+\sum_{i=2}^{m}\mu_i,\nonumber\\
    \lambda_{m_1}&=&\min\limits_{1\in\overline{\alpha}}\{w_{\overline{\alpha}\overline{\alpha}}\}=
     \mu_1+\sum_{i=n-m+2}^{n}\mu_i,\nonumber
  \end{eqnarray}
   and $\lambda_{2}\geq\cdots\geq\lambda_{C_{n}^{m}}$ the eigenvalues of  the matrix $W$.
   We may assume $N>1$, then from (\ref{5.1}) we see that
  \begin{eqnarray}\label{4.34}
    |u_{ii}|\leq 2C_{0}N,\quad \forall i\in I.
  \end{eqnarray}
  Then
  \begin{eqnarray}\label{4.35}
    \lambda_{i}\leq2mC_{0}N\leq \frac{3mC_{0}}{\sigma}u_{11},\quad \forall 1\leq i\leq C_{n}^{m}.
  \end{eqnarray}

\medskip

  We will consider the following cases.

\medskip

  $\mathbf{Case 1}$.  $\lambda_{m_1}\leq0$.

    It follows from (\ref{2.5}) that
  \begin{eqnarray}
    F^{11}
    &>&S_{k-1}(\lambda|m_1)\nonumber\\
      &\geq&\frac{1}{C_{n}^{m}-k+1}\sum_{i=1}^{m}S_{k-1}(\lambda|i)=\frac{1}{m(C_{n}^{m}-k+1)}\mathcal{F}.\nonumber
  \end{eqnarray}
  Then we have
  \begin{eqnarray}
    F^{ij}P_{ij}&\leq&\big(\beta C_{1}-(A+\sigma N)k_{3}\big)(\mathcal{F}+1)+2C_{0}k_{1}N\mathcal{F}\nonumber\\
    &&-\frac{\beta}{2nm(C_{n}^{m}-k+1)}(\frac{A}{3}+\frac{2\sigma N}{3})\mathcal{F}\nonumber\\
    &<&0.
  \end{eqnarray}
  if we choose $\beta>\frac{6nmk_{1}C_{0}(C_{n}^{m}-k+1)}{\sigma}$ and $A>\frac{\beta C_{1}}{k_{3}}$.

\medskip

  $\mathbf{Case 2}$. $\lambda_{m_1}> 0$, $u_{nn}\geq0$.

  It follows from
  \begin{eqnarray}
    kf=\sum_{i=1}^{n}F^{ii}u_{ii}=\sum_{u_{ii}\geq0}F^{ii}u_{ii}\nonumber
  \end{eqnarray}
  and (\ref{4.32}) that
  \begin{eqnarray}
    F^{ij}P_{ij}\leq \big(\beta C_{1}-(A+\sigma N)k_{3}\big)(\mathcal{F}+1)+k_{1}kf<0,
  \end{eqnarray}
  if we choose $A>\frac{3\beta C_{1}+k_{1}k\max f}{k_{3}}$.

\medskip

  $\mathbf{Case 3}$. $\lambda_{m_1}>0$,  $-\frac{k_{3}}{4k_{1}}u_{11}\leq u_{nn}<0$.

  It follows from
  \begin{eqnarray}
    \sum_{u_{ii}\geq0}F^{ii}u_{ii}+\sum_{u_{ii}<0}F^{ii}u_{ii}=kf\nonumber
  \end{eqnarray}
  that
  \begin{eqnarray}
   -k_{1}\sum_{u_{ii}<0}F^{ii}u_{ii}+ k_{1}\sum_{u_{ii}\geq0}F^{ii}u_{ii}
   &=&
   k_{1}(kf-2\sum_{u_{ii}<0}F^{ii}u_{ii})\nonumber\\
    &\leq&k_{1}kf-2k_{1}u_{nn}\mathcal{F}\nonumber\\
    &\leq&k_{1}kf+(\frac{2A}{3}+\frac{\sigma N}{2})k_{3}\mathcal{F}
  \end{eqnarray}
  Similarly we choose $A>\frac{3(\beta C_{1}+k_{1}k\max f)}{k_{3}}$ to get
  \begin{eqnarray}
    F^{ij}P_{ij}<0.
  \end{eqnarray}

\medskip

  $\mathbf{Case 4}$. $\lambda_{m_1}>0$, $u_{nn}<-\frac{k_{3}}{4k_{1}}u_{11}$, $\lambda_{C_n^m}\leq-\delta_1'u_{11}$, $\delta_1'$
  a small positive constant to be determined later.

  Obviously, we have $\lambda_1\geq\lambda_{m_1}>0$.
   If $u_{11}\geq u_{22}$, then it is easy to see $\lambda_1\geq \lambda_2$. Otherwise, $u_{11}<u_{22}$,
   since $2\leq m\leq \frac{n}{2}$, then we have
  \begin{eqnarray}
    \lambda_1&=&\mu_1+\sum_{i=2}^{m}\mu_i\nonumber\\
    &\geq&\lambda_{m_1}+\mu_2-\mu_n\nonumber\\
    &>&u_{11}\geq\frac{\sigma}{3mC_0}\lambda_2.\label{5.45}
  \end{eqnarray}
 Here we use (\ref{4.35}) in the last inequality. Again we use (\ref{4.35}) to have
  \begin{eqnarray}\label{5.46}
  \lambda_{C_n^m}\leq-\frac{\sigma\delta_1'}{3mC_0}\lambda_1.
  \end{eqnarray}
  Now (\ref{5.45}) and (\ref{5.46}) permit us to choose $\delta=\min\{1,\frac{\sigma}{3mC_0}\}=\frac{\sigma}{3mC_0}$ and $\varepsilon=\frac{\sigma\delta_1'}{3mC_0}$ in Proposition \ref{pro5} to give
  \begin{eqnarray}
   F^{11}\geq S_k(\lambda|m_1)\geq c_0S_k(\lambda)=\frac{c_0}{(C_n^m-k+1)}\mathcal{F}.
  \end{eqnarray}
  where $c_0=\min\{\frac{\sigma^4\delta_1'^2}{162m^4(n-2)(n-1)C_0^4}, \frac{\sigma^3\delta_1'^2}{108m^3(n-1)C_0^3}\}$. Similar to the Case 1 we have
  \begin{eqnarray}
    F^{ij}P_{ij}&\leq&\big(\beta C_{1}-(A+\sigma N)k_{3}\big)(\mathcal{F}+1)+2C_{0}k_{1}N\mathcal{F}\nonumber\\
    &&-\frac{c_0\beta}{2n(C_{n}^{m}-k+1)}(\frac{A}{3}+\frac{2\sigma N}{3})\mathcal{F}\nonumber\\
    &<&0,\nonumber
  \end{eqnarray}
   if we choose $\beta>\frac{6nk_{1}C_{0}(C_{n}^{m}-k+1)}{c_0\sigma}$ and $A>\frac{\beta C_{1}}{k_{3}}$.

\medskip

  $\mathbf{Case 5}$. $\lambda_{m_1}>0$, $u_{nn}<-\frac{k_{3}}{4k_{1}}u_{11}$, $\lambda_{C_{n}^{m}}\geq-\delta_{1}'u_{11}$.

  Note that, by (\ref{4.35}),
  \begin{eqnarray}
    \lambda_1\leq\frac{3mC_0}{\sigma}u_{11}.\nonumber
  \end{eqnarray}
  Let $\delta_1'=\frac{3C_0k_3^{k-1}}{(C_n^m)!8^kk_1^{k-1}}$, now we can choose $\delta=\frac{k_3}{4k_1}$ and $\theta_2=\frac{k_3^{k-1}}{4^km^{k-1}k_1^{k-1}(C_n^m)^3}$ in the Proposition \ref{pro7}, such that
  \begin{eqnarray}
    F^{11}\geq S_{k-1}(\lambda|m_1)\geq \theta_2 \sum_{i=1}^{C_n^m} S_k(\lambda|i)=\frac{\theta_2}{m}\mathcal{F}.
  \end{eqnarray}

  Similarly we choose $\beta>\frac{6nmC_0k_1}{\sigma\theta_2}$ and $A>\frac{\beta C_{1}}{k_{3}}$ to get
\begin{eqnarray}
  F^{ij}P_{ij}<0.
\end{eqnarray}

In conclusion, we choose
\begin{eqnarray}
  \beta=\max\{2nk_2+1,\frac{6nmk_{1}C_{0}C_{n}^{m}}{\sigma},\frac{6nk_{1}C_{0}C_{n}^{m}}{c_0\sigma},
\frac{6nmC_0k_1}{\sigma\theta_2}\}.\nonumber
\end{eqnarray}
 Taking $\mu=\min\{\mu_0,\frac{1}{4K_3},\frac{1}{\beta}\}$ and $A>\max \{3\beta C_{2},\frac{3(\beta C_{1}+k_{1}k\max f)}{k_{3}}\}$,
 we obtain $F^{ii}P_{ij}<0$, which contradicts to
  that $P$ attains its minimum in the interior of $\Omega_{\mu}$.
   This implies that $P$ attains its minimum on the boundary $\partial \Omega_{\mu}$.

On $\partial \Omega$, it is easy to see
\begin{eqnarray}
  P=0.
\end{eqnarray}
On $\partial\Omega_{\mu}\cap\Omega$, we have
\begin{eqnarray}
  P\geq-C_{3}(|u|_{C^{1}}, |\psi|_{C^{0}})+(A+\sigma N)\frac{\mu}{2}\geq0,
\end{eqnarray}
if we take $A=\max \{\frac{2C_{3}}{\mu},3\beta C_{2},\frac{3(\beta C_{1}+k_{1}k\max f)}{k_{3}}\}$. Finally the maximum principle tells us that
\begin{eqnarray}
  P\geq0,\quad\text{in}\quad\Omega_{\mu}.
\end{eqnarray}

Suppose $u_{\nu\nu}(y_{0})=\sup_{\partial\Omega}u_{\nu\nu}>0$, we have
\begin{eqnarray}
  0&\geq&P_{\nu}(y_{0})\nonumber\\
  &\geq&(u_{r\nu}h_{r}+u_{r}h_{r\nu}-\psi_{\nu})-(A+\sigma N)h_{\nu}\nonumber\\
  &\geq&u_{\nu\nu}(y_{0})-C(|u|_{C^{1}}, |\partial\Omega|_{C^{2}}, |\psi|_{C^{2}})-(A+\sigma N).\nonumber
\end{eqnarray}
Then we get
\begin{eqnarray}
  \sup_{\partial\Omega}u_{\nu\nu}\leq C+\sigma N.
\end{eqnarray}
\end{proof}

 In a similar way, we construct the super barrier function as
 \begin{eqnarray}
  \overline{P}(x):=g(x)(Du\cdot Dh(x)-\psi(x))+G(x).
\end{eqnarray}
We have the following lemma.
\begin{lemma}\label{le4.4}
  Fix $\sigma$, if we select $\beta$ large, $\mu$ small, and $A$ large, then
  \begin{eqnarray}
    \overline{P}\leq0,\quad\text{in}\quad\Omega_{\mu}.
  \end{eqnarray}
  Furthermore, we have
  \begin{eqnarray}\label{4.77}
    \inf_{\partial \Omega}u_{\nu\nu}\geq -C-\sigma N,
  \end{eqnarray}
  where constant $C$ depends on $|u|_{C^{1}}$, $|\partial\Omega|_{C^{2}}$ $|f|_{C^2}$ and $|\phi|_{C^{2}}$.
\end{lemma}
\begin{proof}
   We assume $\overline{P}(x)$ attains its maximum point $x_{0}$ in the interior of $\Omega_{\mu}$. Differentiate $\overline{P}$ twice to  obtain
  \begin{eqnarray}
    \overline{P}_{i}=g_{i}(u_{r}h_{r}-\psi)+g(u_{ri}h_{r}+u_{r}h_{ri}-\psi_{i})+G_{i},
  \end{eqnarray}
  and
  \begin{eqnarray}
   \overline{P}_{ij}&=&g_{ij}(u_{r}h_{r}-\psi)+g_{i}(u_{rj}h_{r}+u_{r}h_{rj}-\psi_{j})\\
    &&+g_{j}(u_{ri}h_{r}+u_{r}h_{ri}-\psi_{i})+g(h_{rij}h_{r}+u_{ri}h_{rj}\nonumber\\
    &&+u_{rj}h_{ri}+u_{r}h_{rij}-\psi_{ij})+G_{ij}.\nonumber
  \end{eqnarray}

  As before we assume that $(u_{ij})$ is diagonal at $x_{0}$, so are $W$ and $(F_{ij})$.
   We choose $\mu=\min\{\frac{1}{4K_1},\frac{1}{\beta}\}$ so that $|\beta h|\leq \beta\frac{\mu}{2}\leq\frac{1}{2}$.
  By a straight computation we obtain
  \begin{eqnarray}\label{4.81}
    F^{ij}\overline{P}_{ij}&=&F^{ii}g_{ii}(u_{r}h_{r}-\psi)+2F^{ii}g_{i}(u_{ii}h_{i}+u_{r}h_{ri}-\psi_{i})\nonumber\\
     &&+g F^{ii}(u_{rii}h_{r}+2u_{ii}h_{ii}+u_{r}h_{rii}-\psi_{ii})+(A+\sigma N)F^{ii}h_{ii} \nonumber\\
    &\geq& -\big(\beta C_1-(A+\sigma N)k_{3}\big)(\mathcal{F}+1)\\
   && -2\beta F^{ii}u_{ii}h_{i}^{2}+2gF^{ii}u_{ii}h_{ii},\nonumber
  \end{eqnarray}
where$C_1=C_{1}(|u|_{C^{1}}, |\partial \Omega|_{C^{31}}, |\phi|_{C^{2}}, |f|_{C^{1}}, n)$.

  We divide indexes $I=\{1, 2, \cdots, n\}$ into two sets in the following way,
  \begin{eqnarray}
    B=\{i\in I | |\beta h_{i}^{2}|<\frac{k_1}{2}\},\nonumber\\
    G =I\backslash B =\{i\in I | |\beta h_{i}^{2}|\geq\frac{k_1}{2}\},\nonumber
  \end{eqnarray}
  where $k_1$ is a positive number depends on $|\partial\Omega|_{C^2}$ and $K_3$ such that $|D^2h|_{C^0}\leq\frac{k_1}{2}$.

  For $i\in G$, by $\overline{P}_{i}(x_{0})=0$, we get
  \begin{eqnarray}
    u_{ii}=-\frac{A+\sigma N}{g}+\frac{\beta(u_{r}h_{r}-\psi)}{g}-\frac{u_{r}h_{ri}-\psi_{i}}{h_{i}}.
  \end{eqnarray}
  Because $|h_{i}^{2}|\geq\frac{k_1}{2\beta}$ , we have
  \begin{eqnarray}
    |\frac{\beta(u_{r}h_{r}-\psi)}{g}-\frac{u_{r}h_{ri}-\psi_{i}}{h_{i}}|\leq\beta C_{2}(k_1,|u|_{C^{1}},
    |\partial \Omega|_{C^{2}}, |\psi|_{C^{1}}).\nonumber
  \end{eqnarray}
  Then let $A\geq3\beta C_{2}$, we have
  \begin{eqnarray}
 -\frac{4A}{3}-\sigma N\leq u_{ii}\leq -\frac{A}{3}-\frac{2\sigma N}{3},\quad \forall i\in G
  \end{eqnarray}
  We choose $\beta\geq2nk_1+1$ to let $|h_{i}^{2}|\leq\frac{1}{4n}$ for $i\in B$. Because $\frac{1}{2}\leq|Dh|\leq 2$, there is a $i_{0}\in G$, say $i_{0}=1$, such that
  \begin{eqnarray}
    h_{1}^{2}\geq \frac{1}{4n}\nonumber
  \end{eqnarray}

  It follows that
  \begin{eqnarray}\label{4.89}
    -2\beta \sum_{i\in I}F^{ii}u_{ii}h_{i}^{2}&=&-2\beta\sum_{i\in G}F^{ii}u_{ii}h_{i}^{2}-2\beta \sum_{i\in B}F^{ii}u_{ii}h_{i}^{2}\nonumber\\
    &\geq& -2\beta F^{11}u_{11}h_{1}^{2}-2\beta\sum_{u_{ii}\geq0}F^{ii}u_{ii}h_{i}^{2}\nonumber\\
    &\geq&-\frac{\beta F^{11}u_{11}}{2n}-k_1\sum_{u_{ii}\geq0}F^{ii}u_{ii}.
  \end{eqnarray}
  and
  \begin{eqnarray}\label{4.90}
    2g\sum_{i\in I}F^{ii}u_{ii}h_{ii}&=&2g\sum_{u_{ii}\geq 0}F^{ii}u_{ii}h_{ii}+2g\sum_{u_{ii}<0}F^{ii}u_{ii}h_{ii}\nonumber\\
    &\geq&-k_1\sum_{u_{ii}\geq0}F^{ii}u_{ii}+2k_{1}\sum_{u_{ii}<0}F^{ii}u_{ii}.
  \end{eqnarray}
  Plug (\ref{4.89}) and (\ref{4.90}) into (\ref{4.81}) to get
  \begin{eqnarray}\label{4.91}
    F^{ii}\overline{P}_{ij}&\geq&
    \big((A+\sigma N)k_{3}-\beta C_{1}\big)(\mathcal{F}+1)-\frac{\beta}{2n}F^{11}u_{11}\nonumber\\
    &&-2k_{1}\sum_{u_{ii}\geq0}F^{ii}u_{ii}+2k_{1}\sum_{u_{ii}<0}F^{ii}u_{ii}.
  \end{eqnarray}

Denote $u_{22}\geq\cdots\geq u_{nn}$, and $\mu_i=u_{ii}\ (1\leq i\leq n)$ for simplicity. We also denote
  $\lambda_1\geq\lambda_{2}\geq\cdots\geq\lambda_{C_{n}^{m}}$ the eigenvalues of  the matrix $W$
  , and
  \begin{eqnarray}
   \lambda_{k}&=&\sum_{l=1}^{m}\mu_{i_l},\nonumber\\
   &\geq&\mu_{i_1}+\sum_{i=n-m+2}^{n}\mu_i,
   \quad\text{for}\ \mu_{i_1}\geq\cdots\geq\mu_{i_m},\nonumber\\
    \lambda_{m_1}&=&\min\limits_{1\in\overline{\alpha}}\{w_{\overline{\alpha}\overline{\alpha}}\}=
     \mu_1+\sum_{i=n-m+2}^{n}\mu_i.\nonumber
  \end{eqnarray}
   As before, assume $N\geq1$, from (\ref{5.1}) we have
  \begin{eqnarray}\label{5.28}
    |u_{ii}|\leq 2C_{0}N,\quad \forall i\in I.
  \end{eqnarray}

 Because $u_{11}<0$, and from (\ref{2.2}) in Proposition \ref{pro3}, we have $\lambda_k>0$,
 then $\mu_{i_1}>0$. It follows that $\lambda_k\geq\lambda_{m_1}$. Using (\ref{2.1})
 and (\ref{2.2}) again,
 we obtain
 \begin{eqnarray}
   F^{11}>S_{k-1}(\lambda|m_1)\geq S_{k-1}(\lambda|k)\geq C(n,m,k) \mathcal{F}.
 \end{eqnarray}
Similarly we choose  $\beta=\frac{6nk_{1}C_{0}}{\sigma C(n,m,k)}+2nk_1+1$
and $A>\frac{\beta C_{1}}{k_{3}}$
 to get
\begin{eqnarray}
  F^{ij}\overline{P}_{ij}>0.
\end{eqnarray}
 This contradicts to that $P$ attains its maximum in the interior of $\Omega_{\mu}$.
 This contradiction implies that $P$
 attains its maximum on the boundary $\partial \Omega_{\mu}$.


On $\partial \Omega$, it is easy to see
\begin{eqnarray}
  \overline{P}=0.\nonumber
\end{eqnarray}
On $\partial\Omega_{\mu}\cap\Omega$, we have
\begin{eqnarray}
  \overline{P}\leq C_{3}(|u|_{C^{1}}, |\psi|_{C^{0}})-(A+\sigma N)\frac{\mu}{2}\leq0,\nonumber
\end{eqnarray}
if we take $A=\frac{2C_{3}}{\mu}+\frac{\beta C_{1}}{k_{3}}+1$.
 Finally the maximum principle tells us that
\begin{eqnarray}
  \overline{P}\leq0,\quad\text{in}\quad\Omega_{\mu}.
\end{eqnarray}

Suppose $u_{\nu\nu}(y_{0})=\inf_{\partial\Omega}u_{\nu\nu}$, we have
\begin{eqnarray}
  0&\leq&P_{\nu}(y_{0})\nonumber\\
  &\leq&(u_{r\nu}h_{r}+u_{r}h_{r\nu}-\psi_{\nu})+(A+\sigma N)h_{\nu}\nonumber\\
  &\leq&u_{\nu\nu}(y_{0})+C(|u|_{C^{1}}, |\partial\Omega|_{C^{2}}, |\psi|_{C^{2}})+(A+\sigma N).
\end{eqnarray}
Then we get
\begin{eqnarray}
  \inf_{\partial\Omega}u_{\nu\nu}\geq -C-\sigma N.
\end{eqnarray}
\end{proof}

Then we prove Theorem \ref{th5.2} immediately.
\begin{proof}[\bf Proof of Theorem \ref{th5.2}]
 We choose $\sigma=\frac{1}{2}$ in Lemma \ref{le4.3} and \ref{le4.4}, then
\begin{eqnarray}
  \sup_{\partial\Omega}|u_{\nu\nu}|\leq C.\label{5.10}
\end{eqnarray}
Combining (\ref{5.10}) with (\ref{5.1}) in Lemma \ref{th4.4}, we obtain
\begin{eqnarray}
  \sup_{\overline{\Omega}}|D^2u|\leq C.
\end{eqnarray}
\end{proof}


\section{Existence of the Neumann boundary problem}\label{sec6}
We use the  method of continuity to prove the existence theorem for the Neumann problem (\ref{eq1}) and(\ref{eq2}).
\begin{proof}[\textbf{Proof of Theorem \ref{th1.1} and \ref{th0}}]
Consider a family of equations with parameter $t$,
\begin{equation}
\left\{
\begin{aligned}\label{eq6}
 &S_{k}(W)=tf+(1-t)\frac{(C_n^m)!m^k}{(C_n^m-k)!k!},\quad \text{in}\ \Omega,\\
 &u_{\nu}=-au+tb+(1-t)(x\cdot\nu+\frac{a}{2}x^2),\quad \text{on}\ \partial\Omega.
\end{aligned}
\right.
\end{equation}
From Theorem \ref{th1}, \ref{th3}, \ref{th4}, \ref{th5.1} and \ref{th5.2}, we get a glabal $C^2$ estimate independent of $t$
for the equation (\ref{eq6}) in both cases of Theorem \ref{th1.1} and Theorem \ref{th0}.
 It follows that the equation (\ref{eq6}) is uniformly elliptic.
 Due to the concavity of $S_k^{\frac{1}{k}}(W)$ with respect to $D^2u$ (see \cite{cns2}),
 we can get the global H\"older estimates of second derivatives following the arguments
in \cite{lt}, that is, we can get
\begin{eqnarray}
  |u|_{C^{2,\alpha}}\leq C,
\end{eqnarray}
where $C$ depends only on $n$, $m$, $k$, $|u|_{C^{1}}$,$|f|_{C^{2}}$,$\min f$, $|\phi|_{C^{3}}$ and $\Omega$.
 It is easy to see that $\frac{1}{2}x^2$ is a $k$-admissible solution to (\ref{eq6}) for $t=0$.
 Applying the method of continuity (see \cite{gt}, Theorem 17.28), the existence of the classical
solution holds for $t=1$. By the standard regularity theory of uniformly elliptic partial differential
equations, we can obtain the higher regularity.
\end{proof}

\end{document}